\documentclass{amsart}

\usepackage{amsmath, amssymb}
\usepackage{amsthm}
\usepackage{enumerate}
\usepackage{mathtools}

\usepackage{tikz-cd}

\DeclarePairedDelimiter\floor{\lfloor}{\rfloor}

\usepackage[shortlabels]{enumitem}
\setlist[enumerate]{label=\rm{(\arabic*)}}
\setlist[enumerate,2]{label=\rm({\it\roman*})}
\setlist[itemize]{label=\raisebox{0.25ex}{\tiny$\bullet$}}
\usepackage[backref, colorlinks, linktocpage, citecolor = blue, linkcolor = blue]{hyperref}

\newcommand{\GL}{\operatorname{GL}}

\newcommand{\aff}{\operatorname{Aff}}

\newcommand{\p}{\operatorname{\mathbb{P}}}
\newcommand{\C}{\operatorname{\mathbb{C}}}
\newcommand{\Z}{\operatorname{\mathbb{Z}}}
\newcommand{\N}{\operatorname{\mathbb{Z^+}}}

\newcommand{\G}{\operatorname{H}}
\newcommand{\s}{\mathcal{S}}

\newcommand{\A}{\operatorname{\mathbb{A}}}

\newcommand{\id}{\operatorname{id}}

\newcommand{\norm}{\operatorname{Norm}}

\newcommand{\aut}{\operatorname{Aut}}

\newcommand{\J}{\mathcal{J}}

\newcommand{\Bir}{\operatorname{Bir}}
\newcommand{\Cr}{\operatorname{Cr}}

\newcommand{\PGL}{\operatorname{PGL}}
\newcommand{\SL}{\operatorname{SL}}
\newcommand{\T}{{T}}

\newcommand{\SO}{\operatorname{SO}}

\newcommand{\W}{\operatorname{W}}
\newcommand{\exc}{\operatorname{Exc}}

\theoremstyle{plain}
\newtheorem{theorem}{Theorem}
\newtheorem{lemma}[theorem]{Lemma}
\newtheorem{proposition}[theorem]{Proposition}
\newtheorem{corollary}[theorem]{Corollary}
\theoremstyle{definition}
\newtheorem*{definition}{Definition}
\newtheorem{example}[theorem]{Example}
\theoremstyle{remark}
\newtheorem{remark}{Remark}

\numberwithin{theorem}{section}
\title{On homomorphisms between Cremona groups}
\author{Christian Urech}

\address{Mathematisches Institut\\ Universit\"at Basel\\ 4051 Basel\\ Switzerland}
\address{IRMAR\\ Universit\'e de Rennes 1\\ 35042 Rennes\\ France}
\email{christian.urech@unibas.ch}
\thanks{The author gratefully acknowledges support by the Swiss National Science Foundation Grant "Birational Geometry"  PP00P2 128422 /1 as well as by the Geldner-Stiftung, the FAG Basel, the Janggen P\"ohn-Stiftung and the State Secretariat for Education,
Research and Innovation of Switzerland}

\begin{document}
\maketitle
\begin{abstract}
We look at algebraic embeddings of the Cremona group in $n$ variables $\Cr_n(\C)$ to the groups of birational transformations $\Bir(M)$ of an algebraic variety $M$. First we study geometrical properties of an example of an embedding of $\Cr_2(\C)$ into $\Cr_5(\C)$ that is due to Gizatullin.  In a second part, we give a full classification of all algebraic embeddings of $\Cr_2(\C)$ into $\Bir(M)$, where $\dim(M)=3$ and generalize this result partially to algebraic embeddings of $\Cr_n(\C)$ into $\Bir(M)$, where $\dim(M)=n+1$, for arbitrary $n$. In particular, this yields a classification of all algebraic $\PGL_{n+1}(\C)$-actions on smooth projective varieties of dimension $n+1$ that can be extended to rational actions of $\Cr_{n}(\C)$.
\end{abstract}
\tableofcontents

\section{Introduction and statement of the results}

\subsection{Cremona groups} Let $M$ be a complex variety and $\Bir(M)$ the group of birational transformations of $M$. Denote by $\p^n=\p^n_{\C}$ the complex projective space of dimension $n$. The group 
\[
\Cr_n:=\Bir(\p^n)
\] 
is called the {\it Cremona group}. In this paper we are interested in group homomorphisms from $\Cr_n$ to $\Bir(M)$. In particular, we will study an embedding of $\Cr_2$ into $\Cr_5$ that was described by Gizatullin \cite{MR1714823} and consider the case, where $\dim(M)=n+1$.

A birational transformation $A\colon M\dashrightarrow N$  between varieties $M$ and $N$ induces an isomorphism $\Bir(M)\to\Bir(N)$ by conjugating elements of $\Bir(M)$ with $A$. Two homomorphisms $\Phi\colon\Bir(M)\to\Bir(N_1)$ and $\Psi\colon\Bir(M)\to\Bir(N_2)$ are called {\it conjugate} if there exists a birational transformation $A\colon N_1\dashrightarrow N_2$ such that $\Psi(g)=A\circ \Phi(g)\circ A^{-1}$ for all $g\in\Bir(M)$.

\begin{example}\label{product}
Assume that a variety $M$ is birationally equivalent to $\p^n\times N$ for some variety $N$. The standard action on the first factor yields an injective homomorphism of $\Cr_n$\ into $\Bir(\p^n\times N)$ and therefore also into $\Bir(M)$. We call embeddings of this type {\it standard embeddings}.
In particular, we obtain in that way for all nonnegative integers $m$ an injective homomorphism $\Cr_n\to\Bir(\p^n\times\p^m)$.
\end{example}

\begin{example}\label{stablyrational}
 A variety $M$ is called {\it stably rational} if there exists a $n$ such that $M\times\p^n$ is rational. There exist varieties of dimension larger than or equal to $3$ that are stably rational but not rational (see \cite{MR786350}).
We will see that two standard embeddings $f_1\colon\Cr_n\to\Bir(\p^n\times N)$ and $f_2\colon\Cr_n\to\Bir(\p^n\times M)$ are conjugate if and only if $N$ and $M$ are birationally equivalent  (Lemma \ref{birstab}). So every class of birationally equivalent stably rational varieties of dimension $k$ defines a different conjugacy class of injective homomorphisms $\Cr_n\to\Bir(\p^m)$ for $m=n+k$.
\end{example}

\subsection{Notation and subgroups of $\Cr_n$}
 If we fix homogeneous coordinates $[x_0:\dots:x_n]$ of $\p^n$, every element $f\in\Cr_n$ can be described by homogeneous polynomials of the same degree $f_0,\dots, f_n\in\C[x_0,\dots,x_n]$ without non-constant common factor, such that
\[
f([x_0:\dots:x_n])=[f_0:\dots:f_n].
\]
The {\it degree} of $f$ is the degree of the $f_i$.

With respect to affine coordinates $[1:X_1:\dots:X_n]=(X_1,\dots, X_n)$, we have 
\[
f(X_1,\dots, X_n)=(F_1,\dots, F_n),
\]
where $F_i(X_1,\dots,X_n)=f_i(1,X_1,\dots,X_n)/f_0(1,X_1,\dots,X_n)\in\C(X_1,\dots, X_n)$.

An important subgroup of $\Cr_n$ is the automorphism group 
\[
\aut(\p^n)\simeq\PGL_{n+1}(\C).
\]
The $n$-dimensional subgroup of $\aut(\p^n)$ consisting of diagonal automorphisms will be denoted by ~$D_n$.

Let $A=(a_{ij})\in M_n(\Z)$ be a matrix of integers. The matrix $A$ determines a rational self map of the affine space 
\[
f_A=(x_1^{a_{11}}x_2^{a_{12}}\cdots x_n^{a_{1n}}, x_1^{a_{21}}x_2^{a_{22}}\cdots x_n^{a_{2n}},\dots, x_1^{a_{n1}}x_2^{a_{n2}}\cdots x_n^{a_{nn}}).
\] 
We have $f_A\circ f_B=f_{AB}$ for $A,B\in M_n(\Z)$. One observes that $f_A$ is a birational transformation if and only if $A\in\GL_n(\Z)$. This yields an injective homomorphism $\GL_n(\Z)\to\Cr_{n}$ whose image we call the {\it Weyl group} and denote it by $\W_n$. This terminology is justified by the fact that the normalizer of $D_n$ in $\Cr_n$ is the semidirect product $\norm_{\Cr_n}(D_n)=D_n\rtimes\W_n$.
Note that $D_n\rtimes\W_n$ is the automorphism group of $(\C^*)^n$. Sometimes, $\W_n$ is also called the {\it group of  monomial transformations}.

The Cremona group $\Cr_n$ contains $\aut(\A^n)$, the group of polynomial automorphisms of the affine space $\A^n$. We always consider the embedding of $\aut(\A^n)$ into $\Cr_n$ by considering the affine coordinates given by $x_0\neq 0$.

\subsection{Previous results}

The well known theorem of Noether and Castelnuovo (see for example \cite{MR1874328}) states that over an algebraically closed field $k$ the Cremona group in two variables is generated by $\PGL_3(k)$ and the standard quadratic involution
\[
\sigma:=[x_1x_2:x_0x_2:x_0x_1]\in\W_2.
\]

Results of Hudson and Pan (\cite{hudson1927}, \cite{MR1686984}) show that for $n\geq 3$ the Cremona group $\Cr_n$ is not generated by $\PGL_{n+1}(\C)$ and $\W_n$. Let
\[
\G_n:=\left< \PGL_{n+1}(\C), \W_n\right>.
\] 
Blanc and Hed\'en studied the subgroup $G_n$ of $\Cr_n$ generated by $\PGL_{n+1}(\C)$ and the element $\sigma_n:=[x_0^{-1}:\dots:x_n^{-1}]$ (\cite{Blanc:2014aa}). In particular, they show that $G_n$ is strictly contained in $\G_n$ if and only if $n$ is odd. Further results about the group structure of $G_n$ can be found in \cite{Deserti:2014kq}.

Let $\gamma\colon \C\to \C$ be an automorphism of fields. By acting on the coordinates, $\gamma$ induces a bijective map $\Gamma\colon \p^n\to\p^n$. Conjugation with $\Gamma$ yields a group automorphism of $\Cr_n$ that preserves degrees. Observe that we obtain the image of $g\in\Cr_n$ by letting $\gamma$ operate on the coefficients of $g$. By abuse of notation we denote this automorphism by $\gamma$ as well. In \cite{MR2278755} D\'eserti showed that all automorphisms of $\Cr_2$ are inner up to such  field automorphisms. A generalization of this result is the following theorem by Cantat:

\begin{theorem}[\cite{MR3230847}]\label{thmserge}
Let $M$ be a smooth projective variety of dimension $n$ and $r\in\N$. Let 
\[
\rho\colon \PGL_{r+1}(\C)\to\Bir(M)
\]
be a non-trivial group homomorphism. Then $n\geq r$ and if $n=r$ then $M$ is rational and there exists an automorphism of fields $\gamma\colon\C\to\C$ such that $\rho\circ\gamma$ is conjugate to the standard embedding of $\PGL_{n+1}(\C)$ into $\Cr_n$.
\end{theorem}
 
In the Appendix we will prove two corollaries of Theorem \ref{thmserge} that show some implications of this result to group endomorphisms of $\Cr_n$.

\subsection{Algebraic homomorphisms}
We call a group homomorphism $\Psi\colon\Cr_n\to\Bir(M)$ {\it algebraic} if its restriction to $\PGL_{n+1}(\C)$ is an algebraic morphism. The algebraic structure of $\Bir(M)$ and some properties of algebraic homomorphisms will be discussed in Section \ref{algebraichom}. Recall that an element $f\in\Cr_n$ is called {\it algebraic}, if the sequence $\{\deg(f^n)\}_{n\in\N}$ is bounded. 

\subsection{Reducibility} 
\begin{definition}\label{redu}
Let $M$ be a variety $\varphi_M\colon \Cr_n\to\Bir(M)$ a non-trivial algebraic group homomorphism. We say that $\varphi_M$ is {\it reducible} if there exists a variety $N$ such that $0<\dim(N)<\dim(M)$ and an algebraic homomorphism $\varphi_N\colon \Cr_n\to \Bir(N)$ together with a dominant rational map $\pi\colon M\dashrightarrow N$ that is $\Cr_2$-equivariant with respect to the rational actions induced by $\varphi_M$ and $\varphi_N$ respectively. 
\end{definition}
\begin{remark}
In \cite{MR2565534}, Zhang uses the terminology {\it primitive action} for irreducible actions in the sense of Definition \ref{redu}; in \cite{MR2026896}, Cantat says that an action {\it admits a non-trivial factor} if it is reducible.
\end{remark}
Note that if we look at the induced action of $\Cr_n$ on the function field $\C(M)$ of $M$, reducibility is equivalent to the existence of a $\Cr_n$-invariant function field $\C(N)\subset\C(M)$.

\subsection{An example by Gizatullin}
In \cite{MR1714823}, Gizatullin looks at the following question: Let $\psi\colon\PGL_3(\C)\to\PGL_{n+1}(\C)$ be a linear representation. Does $\psi$ extend to a homomorphism $\Psi\colon\Cr_2\to\Cr_n$? He shows that the linear representations given by the action of $\PGL_3(\C)$ on conics, cubics and quartics can be extended to homomorphisms from $\Cr_2$ to $\Cr_5$, $\Cr_9$ and $\Cr_{14}$, respectively. 

In Section \ref{gizatullinexample} we study in detail some geometrical properties of the homomorphism
\[
\Phi\colon\Cr_2\to\Cr_5
\]
that was described by Gizatullin; by construction, the restriction of $\Phi$ to $\PGL_3$ yields the linear representation $\varphi\colon\PGL_3(\C)\to\PGL_6(\C)$ given by the action of $\PGL_3(\C)$ on plane conics.
Among other things, we prove the following:

\begin{theorem}\label{gizatullintheorem} Let $\Phi\colon\Cr_2\to\Cr_5$ be the Gizatullin homomorphism. Then the following is true:
\begin{enumerate}
\item The group homomorphism $\Phi$ is injective and irreducible.
\item The rational action of $\Cr_2$ on $\p^5$ that is induced by $\Phi$ preserves the Veronese surface $V$ and its secant variety $S\subset\p^5$ and induces rational actions of $\Cr_2$ on $V$ and $S$.
\item The Veronese embedding $v\colon\p^2\to\p^5$ is $\Cr_2$-equivariant with respect to the standard rational action on $\p^2$.
\item The surjective secant morphism $s\colon\p^2\times\p^2\to S\subset\p^5$ (see Section \ref{geometry}) is  $\Cr_2$-equivariant with respect to the diagonal action of $\Cr_2$ on $\p^2\times\p^2$.
\item The rational action of $\Cr_2$ on $\p^5$ preserves a volume form on $\p^5$ with poles of order three along the secant variety $S$.
\item The group homomorphism $\Phi$ sends the group of polynomial automomorphisms $\aut(\A^2)\subset\Cr_2$ to $\aut(\A^5)$.
\end{enumerate}
\end{theorem}

Note that the injectivity of $\Phi$ follows from (3); in Section \ref{irredsection} irreducibility is proved. Part (2) - (4) of Theorem \ref{gizatullintheorem} will be proved in Section \ref{geometry}, part (5) in Section~\ref{volume} and part (6) in Section \ref{polauto}

The representation $\varphi^\vee$ of $\PGL_3$ into $\PGL_6$ given by $\psi\circ\alpha$, where $\alpha$ is the algebraic homomorphism $g\mapsto {}^tg^{-1}$, is conjugate in $\Cr_5$ to the representation $\varphi$. This conjugation yields the embedding $\Phi^\vee\colon \Cr_2\to\Cr_5$, whose image preserves the secant variety $S$ as well and induces a rational action on it. As the secant variety $S$ is rational, $\Phi$ and $\Phi^\vee$ induce two non-standard embeddings of $\Cr_2$ into $\Cr_4$, which we denote by $\Psi_1$ and $\Psi_2$ respectively.  In Section \ref{inducedsection} we prove the following:

\begin{proposition}\label{inducedembedd}
The two embeddings $\Psi_1, \Psi_2\colon\Cr_2\to\Cr_4$ are not conjugate in $\Cr_4$; moreover they are irreducible and therefore not conjugate to the standard embedding. 
\end{proposition}
Proposition \ref{inducedembedd} shows in particular that there exist at least three different embeddings of $\Cr_2$ into $\Cr_4$.

Since $\Phi$ is algebraic, the images of algebraic elements under $\Phi$ are algebraic again (see Proposition \ref{alghom}). Calculation of the degrees of some examples suggests that $\Phi$ might even preserve the degrees of all elements in $\Cr_2$. However, we were only able to prove the following (Section \ref{polauto}):

\begin{theorem}\label{degreeaut}
Let $\Phi\colon\Cr_2\to\Cr_5$ be the Gizatullin-embedding. Then 
\begin{enumerate}
\item for all elements $f\in\Cr_2$ we have $\deg(f)\leq\deg(\Phi(f))$,
\item for all $g\in\aut(\A^2)\subset\Cr_2$ we have $\deg(g)=\deg(\Phi(g))$.
\end{enumerate}
\end{theorem}

The image of the Weyl group $\W_2$ under $\Phi$ is not contained in the Weyl group~$\W_5$. {\it More generally, it can be shown that there exists no algebraic homomorphism from $\Cr_2$ to $\Cr_5$ that preserves automorphisms, diagonal automorphisms and the Weyl group} (see \cite{longversion}).

\subsection{Algebraic embeddings in codimension 1} 
In Section \ref{pglnactions} and Section \ref{extension} we look at algebraic homomorphisms $\Cr_n\to\Bir(M)$ in the case where $M$ is a smooth projective variety of dimension $n+1$ for $n\geq 2$.

\begin{example}\label{pntimesc}
For all curves $C$ of genus $\geq 1$, the variety $\p^n\times C$ is not rational and there exists the standard embedding $\Psi_C\colon\Cr_n\to\Bir(\p^n\times C)$.
\end{example}

\begin{example}\label{defl}
$\Cr_n$ acts rationally on the total space of the canonical bundle of $\p^n$
\[
K_{\p^n}\simeq\mathcal{O}_{\p^n}(-(n+1))\simeq\bigwedge^n (\T\p^n)^\vee
\]
by $f(p, \omega)=(f(p), \omega\circ (df_p)^{-1}),$
where $p\in\p^n$ and $\omega\in \bigwedge^n(\T_p\p^2)^\vee$. This action extends to the projective completion 
\[
F_1:=\p(\mathcal{O}_{\p^n}\oplus\mathcal{O}_{\p^n}(-(n+1))).
\]

More generally, we obtain an action of $\Cr_n$ on the total space of the bundle $K_{\p^n}^{\otimes l}\simeq\mathcal{O}_{\p^n}(-(n+1)l)$ and on its projective completion 
\[
F_l:=\p(\mathcal{O}_{\p^n}\oplus\mathcal{O}_{\p^n}(-l(n+1))
\]
for all $l\in\Z_{\geq 0}$.  This yields a countable family of injective homomorphisms 
\[
\Psi_{l}\colon \Cr_n\to \Bir(F_l).
\]

We can choose affine coordinates $(x_1,\dots,x_n,x_{n+1})$  of $F_{l}$ such that $\Psi_{l}$ is given by
 \[
 \Psi_{l}(f)(x_1,\dots,x_n,x_{n+1})=(f(x_1,\dots,x_n), J(f(x_1,\dots,x_n))^{-l} x_{n+1}).
 \]
  Here, $J(f(x_1,\dots,x_n))$ denotes the determinant of the Jacobian of $f$ at the point $(x_1,\dots,x_n)$.
Observe that $\Psi_0$ is conjugate to the standard embedding.
\end{example}

\begin{example}\label{defb}
Let $\p(\T\p^2)$ be the total space of the fiberwise projectivisation of the tangent bundle over $\p^2$. Then $\p(\T\p^2)$ is rational and there is an injective group homomorphism
 \[
 \Psi_B\colon \Cr_2\to \Bir(\p(\T\p^2))
 \] 
 defined by 
$ \Psi_B(f)(p, v)\coloneqq(f(p), \p (df_p)(v)).$ 
 Here, $\p (df_p)\colon \p \T_p\to \p \T_{f(p)}$ defines the projectivisation of the differential $df_p$ of $f$ at the point $p\in\p^2$. 
\end{example}

\begin{example}\label{grass}
The Grassmannian of lines in the projective 3-space $\mathbb{G}(1,3)$ is a rational variety of dimension 4 with a transitive algebraic $\PGL_4(\C)$-action. This action induces an algebraic embedding of $\PGL_4(\C)$ into $\Cr_4$. In Proposition \ref{nograss} we will show that the image of this embedding does not lie in any subgroup isomorphic to $\Cr_3$. {\it So no group action of $\PGL_4(\C)$ on $\mathbb{G}(1,3)$ by automorphisms can be extended to a rational action of $\Cr_3$.}
\end{example}

The classification of $\PGL_{n+1}$-actions on smooth projective varieties of dimension $n+1$ is well known to the experts; in Section  \ref{pglnactions} we study their conjugacy classes. In fact, we will see that Examples \ref{pntimesc} to \ref{grass} describe up to birational conjugation and up to algebraic homomorphisms of $\PGL_{n+1}$ all possible $\PGL_{n+1}$-actions on smooth projective varieties of dimension $n+1$ and that these actions are not birationally conjugate to each other. This yields a classification of algebraic homomorphisms of $\PGL_{n+1}$ to $\Bir(M)$. We will study in Section \ref{extension} how these actions extend to rational actions of $\Cr_n$ on $M$.

\begin{theorem}\label{crmain}
Let $n\geq 2$ and let $M$ be a complex projective variety of dimension $n+1$ and let $\varphi\colon\PGL_{n+1}(\C)\to\Bir(M)$ be a non-trivial algebraic homomorphism, then $\varphi$ is conjugate to one of the embeddings described in Example $\ref{pntimesc}$ to $\ref{defb}$. If $\varphi$ is not conjugate to the action described in Example \ref{grass}, then there exists up to conjugation a unique algebraic homomorphism $\alpha$ of $\PGL_{n+1}(\C)$ such that $\varphi\circ\alpha$ extends to a homomorphism of $\Cr_n$ to $\Bir(M)$. Moreover, this extension is unique if restricted to the subgroup $\G_n=\left<\PGL_{n+1}(\C),\W_n\right>\subset\Cr_n$. 
\end{theorem}

Theorem \ref{crmain} classifies all group homomorphisms $\Psi\colon\G_n\to\Bir(M)$ for projective varieties $M$ of dimension $n+1$ such that the restriction to $\PGL_{n+1}(\C)$ is a morphism. By the theorem of Noether and Castelnuovo, we obtain in particular a full classification of all algebraic homomorphisms from $\Cr_2$ to $\Bir(M)$ for projective varieties $M$ of dimension 3:

\begin{corollary}
Let $M$ be a projective variety of dimension $3$ and $\Psi\colon\Cr_2\to\Bir(M)$ a non-trivial algebraic group homomorphism. Then $\Psi$ is conjugate to exactly one of the homomorphisms described in Example $\ref{pntimesc}$ to $\ref{defb}$. \end{corollary}

The following observations are now immediate:

\begin{corollary}\label{cor3}
Let $M$ be a projective variety of dimension $3$ and $\Psi\colon\Cr_2\to\Bir(M)$ a non-trivial algebraic homomorphism. Then
\begin{enumerate}
\item $\Psi$ is injective.
\item There exists a $\Cr_2$-equivariant rational map $f\colon M\dashrightarrow \p^2$ with respect to the rational action induced by $\Psi$ and the standard action respectively. In particular, all algebraic homomorphisms from $\Cr_2$ to $\Bir(M)$ are reducible.

\item There exists an integer $C_\Psi\in\Z$ such that 
\[
1/C_\Psi\deg(f)\leq \deg(\Psi(f))\leq C_\Psi\deg(f).
\]
\end{enumerate}
\end{corollary}
Note that Part (3) of Corollary \ref{cor3} resembles in some way Theorem \ref{degreeaut}. It seems to be an interesting question how the degree of the image of an element $f\in\Cr_2$ under an algebraic homomorphism is related to the degree of $f$.

\subsection{Acknowledgements}
I thank my PhD advisors J\'er\'emy Blanc and Serge Cantat for their constant support, all the interesting discussions and for the helpful remarks on previous versions of this article.


\section{Algebraic homomorphisms}\label{algebraichom}

In this section we recall some results on the algebraic structure of $\Bir(M)$ and of some of its subgroups and we discuss our notion of algebraic homomorphisms. 

\subsection{The Zariski topology}
We can equip $\Bir(M)$ with the so-called Zariski topology. Let $A$ be an algebraic variety and 
\[
f\colon A\times M\dashrightarrow A\times M
\] 
an $A$-birational map inducing an isomorphism between open subsets $U$ and $V$ of $A\times M$ such that the projections from $U$ and from $V$ to $A$ are both surjective. For each $a\in A$ we obtain therefore an element of $\Bir(M)$ defined by $x\mapsto p_2(f(a,x))$, where $p_2$ is the second projection. Such a map $A\to \Bir(M)$ is called a {\it morphism} or {\it family of birational transformations parametrized by $A$}.

\begin{definition}
The Zariski topology on $\Bir(M)$ is the finest topology such that all morphisms $f\colon A\to \Bir(M)$ for all algebraic varieties $A$ are continuous (with respect to the Zariski topology on $A$).
\end{definition}

The map 
 $\iota\colon\Bir(M)\to\Bir(M),\,x\mapsto x^{-1}$
 is continuous as well as the maps $x\mapsto g\circ x$ and $x\mapsto x\circ g$ for any $g\in\Bir(M)$. This follows from the fact that the inverse of an $A$-birational map as above is again an $A$-birational map as is the right/left-composition with an element of $\Bir(M)$.
The Zariski topology was introduced in \cite{MR0284446} and \cite{serre2008groupe} and studied in \cite{MR3092478}.


\subsection{Algebraic subgroups}\label{algsubgroups}
An algebraic subgroup of $\Bir(M)$ is the image of an algebraic group $G$ by a morphism $G\to \Bir(M)$ that is also an injective group homomorphism. It can be shown that algebraic groups are closed in the Zariski topology and of bounded degree in the case of $\Bir(M)=\Cr_n$. Conversely, closed subgroups of bounded degree in $\Cr_n$ are always algebraic subgroups with a unique algebraic group structure that is compatible with the Zariski topology (see \cite{MR3092478}). 

Let $N$ be a smooth projective variety that is birationally equivalent to $M$. Let $G$ be an algebraic group acting regularly and faithfully on $N$. This yields a morphism $G\to\Bir(M)$, so $G$ is an algebraic subgroup of $\Bir(M)$. On the other hand, a theorem by Weil states that all algebraic subgroups of $\Bir(M)$ have this form.

\begin{theorem}[\cite{MR0074083}, \cite{MR0337963}, \cite{MR1389430}]
Let $G\subset\Bir(M)$ be an algebraic subgroup. Then there exists a smooth projective variety $N$ and a birational map $f\colon M\dashrightarrow N$ that conjugates $G$ to a subgroup of $\aut(N)$ such that the induced action on $N$ is algebraic.
\end{theorem}

It can be shown (see for example, \cite{MR3092478}) that the sets $(\Cr_n)_{\leq d}\subset \Cr_n$ consisting of all birational transformations of degree $\leq d$ are closed with respect to the Zariski topology. So the closure of a subgroup of bounded degree in $\Cr_n$ is an algebraic subgroup and can therefore be regularized. We obtain:

\begin{corollary}
Let $G\subset\Cr_n$ be  a subgroup that is contained in some $(\Cr_n)_{\leq d}$, then there exists a smooth projective variety $N$ and a birational transformation $f\colon\p^n\dashrightarrow N$ such that $fGf^{-1}\subset\aut(N)$.
\end{corollary}

The maximal algebraic subgroups of $\Cr_2$ have been classified together with the rational surfaces on which they act as automorphisms (\cite{enriques1893sui}, \cite{MR2504924}). In dimension 3, a classification for maximal connected algebraic subgroups exists:  \cite{MR683251}, \cite{MR803342}, \cite{MR676586}.

\subsection{Algebraic homomorphisms and continuous homomorphisms}
We defined a group homomorphism from $\Cr_n$ to $\Bir(M)$ to be algebraic if its restriction to $\PGL_{n+1}(\C)$ is a morphism. Note that this is a priori a weaker notion than being continuous with respect to the Zariski topology. It is not clear, whether algebraic homomorphisms are always continuous. However, for dimension 2 we have the following partial result, which will proved in Section \ref{proofalghom}:

\begin{proposition}\label{alghom}
Let $\Phi\colon\Cr_2\to\Bir(M)$ be a homomorphism of groups. The following are equivalent:
\begin{enumerate}
\item $\Phi$ is algebraic.
\item The restriction of $\Phi$ to any algebraic subgroup of $\Cr_2$ is algebraic.
\item The restriction of $\Phi$ to one positive dimensional algebraic subgroup of $\Cr_2$ is algebraic.
\end{enumerate}
\end{proposition}

\subsection{One-parameter subgroups}
A one-parameter subgroup is a connected algebraic group of dimension 1. It is well known (see for example \cite{MR0396773}) that all one-parameter subgroups are isomorphic to either $\C$ or $\C^*$. The group $\C$ is unipotent, the group $\C^*$ semi-simple.

Proposition \ref{onepar} shows that, up to conjugation by birational maps, there exists only one birational action of $\C$ and only one of $\C^*$ on $\p^2$:

\begin{proposition}\label{onepar}
In $\Cr_2$ all one-parameter subgroups isomorphic to $\C$ are conjugate and all one-parameter subgroups isomorphic to $\C^*$ are conjugate.
\end{proposition}

The first part of Proposition \ref{onepar} follows from results in \cite{MR3410471} and \cite{MR2215969} (see also  \cite{MR1474805}). The second part is a special case of Theorem \ref{poptori}. A detailed explanation of the proof can be found in \cite{longversion}.

\begin{theorem}[\cite{MR0200279}, \cite{MR3135700}]\label{poptori}
In $\Cr_n$ all tori of dimension $\geq n-2$ are conjugate to a subtorus of $D_n$. Moreover, two subtori of $D_n$ are conjugate in $\Cr_n$ to each other if and only if they are isomorphic.
\end{theorem}

The following Lemma is a classical result (see for example \cite{MR3228629}):
\begin{lemma}\label{grouplemma}
Let $G$ be a linear algebraic group and $U_{1},\dots, U_{n}$ be algebraic subgroups such that $U_1U_2\cdots U_n=G$. Let $H$ be a linear algebraic group and $\varphi\colon G\to H$ a homomorphism of abstract groups such that $\varphi|_{U_i}$ is a homomorphism of algebraic groups for all $i$. Then $\varphi$ is a homomorphism of algebraic groups.
\end{lemma}


\subsection{Algebraic and abstract group homomorphisms}\label{proofalghom}
Let $G$ and $H$ be algebraic groups that are isomorphic as abstract groups. The question whether $G$ and $H$ are also isomorphic as algebraic groups have been treated in detail in \cite{MR0316587}  (see also \cite{MR0310083} and \cite{deserti:tel-00125492}). We will use the following result:

\begin{proposition}\label{isopgl}
Let $G$ be an algebraic group that is isomorphic to $\PGL_n(\C)$ as an abstract group. Then $G$ is isomorphic to $\PGL_n(\C)$ as an algebraic group. Moreover, for every abstract  isomorphism 
\[
\rho\colon\PGL_n(\C)\to G
\] 
there exists an automorphism of fields $\tau\colon\C\to\C$ such that $\rho\circ\tau$ is an algebraic isomorphism.
\end{proposition}

\begin{remark}
It is well known that the automorphisms of $\PGL_n(\C)$ as an algebraic group are compositions of inner automorphisms and the automorphism
\[
\alpha\colon\PGL_n(\C)\to\PGL_n(\C), \hspace{2mm} g\mapsto {}^tg^{-1}.
\]
\end{remark}

\begin{proof}[Proof of Proposition $\ref{alghom}$]
We first show how (1) implies (2). Let $G$ be an algebraic subgroup of $\Cr_2$. We can assume that $G$ is connected. 
There exist one parameter subgroups $U_1,\dots, U_k\subset G$ such that $U_1\cdots U_k=G$ and there exists a constant $C$ such that every element in $G$ can be written as the product of at most $C$ elements of $U_1\cup U_2\cup\dots\cup U_n$. Since, by Proposition \ref{onepar}, the group $U_i$ is conjugate to a one parameter subgroup of $\PGL_3(\C)$ for all $i$, we obtain that the restriction of $\varphi$ to any of the $U_i$ is an algebraic homomorphism of groups and that $\varphi(G)\subset\Cr_n$ is of bounded degree. Then $\overline{\varphi(G)}\subset\Cr_n$ is an algebraic group. We can now apply Lemma ~\ref{grouplemma} and conclude that the restriction of $\varphi$ to $G$ is a homomorphism of algebraic groups.

Statement (3) follows immediately from statement (2), so it only remains to prove that (3) implies (1).
Let $\varphi\colon\Cr_2\to\Bir(M)$ be a homomorphism of abstract groups and let $G\subset\Cr_2$ be a positive dimensional algebraic subgroup such that the restriction of $\varphi$ to $G$ is a morphism.  Since $G$ is infinite, it contains a one parameter subgroup $U\subset G$. 

Let $U_1,\dots, U_n\subset\PGL_3(\C)$ be unipotent one parameter subgroups such that $U_1\cdots U_n=\PGL_3(\C)$ and $C$ a constant such that every element in $\PGL_3(\C)$ can be written as the product of at most $C$ elements of $U_1\cup U_2\cup\dots\cup U_n$. If $U$ is unipotent, all the subgroups $U_i$ are conjugate to $U$. Hence the restriction of $\varphi$ to $U_i$ is a morphism for all $i$.  The image $\varphi(\PGL_3(\C))\subset\Cr_n$ is of bounded degree, so $\overline{\varphi(\PGL_3(\C)})\subset\Cr_n$ is an algebraic group and with Lemma \ref{grouplemma} it follows that the restriction of $\varphi$ to $\PGL_3(\C)$ is a morphism.

Denote by $D_1\subset\PGL_3(\C)$ the subgroup given by elements of the form $[cx_0: x_1: x_2]$, $c\in\C^*$ and by $T\subset\PGL_3(\C)$ the subgroup of all elements of the form $[x_0:x_1+cx_0:x_2]$, $c\in\C$; we have $D_1\simeq\C^*$ and $T\simeq \C$.
If $U$ is semi-simple, it is, again by Proposition \ref{onepar}, conjugate to $D_1$, hence the restriction of $\varphi$ to $D_1$ is a morphism well. Note that 
\[
T=\{[x_0:x_1+cx_0:x_2]\mid c\in\C\}=\{dgd^{-1}\mid d\in D_1\}\cup \{\id\}
\]
where $g=[x_0:x_1+x_0:x_2]$. We obtain that $\varphi(T)$ is of bounded degree and contained in the algebraic group $\overline{\varphi(T)}\subset\Cr_n$. As $\varphi(T)$ consists of two $\varphi(D_1)$-orbits, it is constructible and therefore closed. We obtain that the images of all unipotent subgroups of $\Cr_2$ under $\varphi$ are algebraic subgroups.The map 
$\varphi(U_1)\times \cdots\times \varphi(U_n)\to\Cr_n$
is a morphism, so its image is a constructible set and therefore closed since it is a group. Hence $
\varphi(\PGL_3(\C))=\varphi(U_1)\cdots\varphi(U_n)$
is an algebraic subgroup. By Proposition \ref{isopgl} it is isomorphic as an algebraic group to $\PGL_3(\C)$ and there exists an automorphism of fields $\tau\colon\C\to\C$ such that $\varphi\circ\tau\colon \PGL_3(\C)\to\PGL_3(\C)$ is an isomorphism of algebraic groups. But since the restriction of $\varphi$ to $T$ is already an algebraic homomorphism, it follows that $\tau$ is the identity.
\end{proof}

\begin{remark}
Proposition  \ref{alghom} shows in particular that algebraic homomorphisms  $\Psi\colon\Cr_2\to\Bir(M)$ send algebraic elements to algebraic elements. This result follows also directly from the fact that a birational transformation $f\in\Cr_2$  of degree $d$ can be written as the product of at most $4d$ linear maps and $4d$ times the standard quadratic involution $\sigma$ (see for example \cite{MR1874328}); we therefore obtain that the sequence $\{\deg(\Phi(f)^n\}$ is bounded if $\{\deg(f^n)\}$ is bounded. 
 \end{remark}


\section{An example by Gizatullin}\label{gizatullinexample}

\subsection{Projective representations of the projective linear group}\label{representation}
The results from representation theory of linear algebraic groups that we use in this section can be found, for example, in \cite{MR1153249}, \cite{MR2265844}. 

\begin{proposition}\label{projrep}
There is a bijection between homomorphisms of algebraic groups from $\SL_n(\C)$ to $\SL_m(\C)$ such that the image of the center is contained in the center and homomorphisms of algebraic groups from $\PGL_n(\C)$ to $\PGL_m(\C)$.
\end{proposition}

From Propostion \ref{projrep} and some elementary representation theory of $\SL_3(\C)$ it follows that $n=6$ is the smallest number such that there exist non-trivial and non-standard homomorphisms of algebraic groups from $\PGL_3(\C)$ to $\PGL_n(\C)$. In fact, up to automorphisms of $\PGL_3(\C)$ there are exactly two non-trivial representations from $\PGL_3(\C)$ to $\PGL_6(\C)$. 

The first one is reducible. Let $\psi'\colon\GL_3\to\GL_6$ be the linear representation given by the diagonal action on $\C^3\times\C^3$; we denote by $\psi\colon\PGL_3(\C)\to\PGL_6(\C)$ its projectivisation. 

The second one is given by the action of $\PGL_3(\C)$ on the space of conics. The latter one can be parametrized by the space $\p M_3$ of symmetric $3\times 3$-matrices up to scalar multiple and is isomorphic to $\p^5$. Let $g\in\PGL_3(\C)$, we define $\varphi(g)\in\PGL_6(\C)$ by $(a_{ij})\mapsto g(a_{ij}) ({}^tg).$

In this section we identify the space of conics with $\p^5$ in the following way:
\[
(a_{ij})\mapsto [a_{00}:a_{11}:a_{22}:a_{12}:a_{02}:a_{01}]
\]
In other words, the conic $C$ given by the zeroes of the equation
\[
F=a_{00}X^2+a_{11}Y^2+a_{22}Z^2+2a_{12}YZ+2a_{02}XZ+2a_{01}XY
\]
is identified with the point 
$ \left[a_{00}:a_{11}:a_{22}:a_{12}:a_{02}:a_{01}\right]\in\p^5.$

Observe that with our definition, $\varphi(g)$ sends the conic $C$ to the conic given by the zero set of the polynomial $F\circ({}^tg)$.

Let 
\[
\alpha\colon \PGL_3(\C)\to\PGL_3(\C)
\] 
be the algebraic automorphism $g\mapsto ({}^tg)^{-1}$. Then $\varphi(\alpha(g))$ maps the conic $C$ to $g(C)$, which is the conic given by the zero set of the polynomial $F\circ g^{-1}$. Accordingly, $\varphi(\alpha(g))\in\PGL_6(\C)$ maps the matrix $(a_{ij})\in M_3$ to $({}^tg)^{-1}(a_{ij})g^{-1}$.

The action of $\PGL_3(\C)$ on $\p^5$ induced by $\varphi$ has exactly three orbits that are characterized by the rank of the corresponding symmetric matrix in $M_3$. Geometrically they correspond to the sets of smooth conics, pairs of distinct lines and double lines. The set of double lines is a surface isomorphic to $\p^2$ and called the {\it Veronese surface}; we denote it by $V$. The set of singular conics $S$ is the secant variety of $V$ and has dimension $4$.

To describe the $\PGL_3(\C)$-orbits with respect to the action induced by $\psi$, consider a point $p=[x_0:x_1:x_2:x_3:x_4:x_5]\in\p^5$. Then $p$ can either be mapped by an element of $\psi(\PGL_3(\C))$ to a point of the form $[a:0:0:b:0:0]$, where $[a:b]\in\p^1$, or to the point $[1:0:0:0:0:1]$ and these points are all in different $\psi(\PGL_3(\C))$-orbits. The stabilizer of $[1:0:0:0:0:1]$ in $\psi(\PGL_3(\C))$ is the subgroup of matrices of the form
\[
\left[\begin{array}{cc}
g&0 \\
0&g
\end{array}\right], \text{ where $g\in\PGL_3(\C)$ has the form} \left[\begin{array}{ccc}
1&a&0 \\
0&b&0\\
0&c&1
\end{array}\right].
\]
Therefore, the orbit of $[1:0:0:0:0:1]$ under $\psi(\PGL_3(\C))$ has dimension $5$. The orbit of a point of the form $[a:0:0:b:0:0]$, on the other hand, has dimension 2. So we have a family parametrized by $\p^1$ of orbits of dimension 2 and one orbit of dimension $5$. In particular, there is no $\psi(\PGL_3(\C))$-invariant subset of dimension ~4.

The following observation is easy but useful. We leave its proof to the reader.

\begin{lemma}\label{exci}
Let $X$ and $Y$ be two projective varieties with biregular actions of a group $G$ and let $f\colon X\dashrightarrow Y$ be a $G$-equivariant rational map. Then the indeterminacy locus $I_f\subset X$ and the exceptional divisor $\exc(f)\subset X$ are $G$-invariant sets.
\end{lemma}

Note that Lemma \ref{exci} implies in particular that all equivariant rational maps with respect to actions without orbits of codimension $\geq 2$ are morphisms.

\begin{lemma}\label{birstab}
Let $M$ and $M'$ be irreducible complex projective varieties such that $M\times\p^n$ et $M'\times\p^n$ are birationally equivalent. Then the standard embeddings 
\[
\Psi\colon\PGL_{n+1}(\C)\to\Bir(\p^n\times M) \text{ and } \Psi'\colon\PGL_{n+1}(\C)\to\Bir(\p^n\times M')
\]
are conjugate if and only if $M$ and $M'$ are birationally equivalent.
\end{lemma}

\begin{proof}
If $M$ and $M'$ are birationally equivalent it follows directly that $\Psi$ and $\Psi'$ are conjugate.
On the other hand, assume that there exists a birational map  
$A\colon \p^n\times M\dashrightarrow \p^n\times M'$ 
that conjugates $\Psi$ to $\Psi'$, i.e.
$A\circ\Psi(g)=\Psi'(g)\circ A$ 
for all $g\in\PGL_{n+1}(\C)$. The images $\Psi(\PGL_{n+1}(\C))$ and $\Psi'(\PGL_{n+1}(\C))$ permute the fibers $\{p\}\times M$, $p\in\p^n$ and $\{p\}\times M'$, $p\in\p^n$ respectively. By Lemma \ref{exci}, no fiber is fully contained in the exceptional locus of $A$.

The fiber 
\[
F:=[1:1:\dots:1]\times M\subset \p^n\times M
\] 
consists of all fixed points of the image of the subgroup of coordinate permutations $\Psi(\s_{n+1})$ and it is isomorphic to $M$. Correspondingly, the fiber
\[
F':=[1:1:\dots:1]\times M'\subset \p^n\times M'
\] 
consists of all fixed points of $\Psi'(\s_{n+1})$ and is isomorphic to $M'$. Hence the strict transform of $F$ under $A$ is $F'$ and we obtain that $M$ and $M'$ are birationally equivalent.
\end{proof}

\begin{proposition}\label{notconj}
Let $\varphi, \psi\colon\PGL_3(\C)\to\PGL_6(\C)$ be the homomorphisms defined in Section $\ref{representation}$. The subgroups $\varphi(\PGL_3(\C))$ and $\psi(\PGL_3(\C))$ are not conjugate in $\Cr_5$.
\end{proposition}

\begin{proof}
Assume that there is an element $f\in\Cr_5$ conjugating $\varphi(\PGL_3(\C))$ to $\psi(\PGL_3(\C))$. Note that $\p^5$ has no $\psi(\PGL_3(\C))$-invariant subset of dimension 4. Hence, by Lemma \ref{exci}, $f$ must be a birational morphism and therefore an automorphism. But this isn't possible since the action of $\varphi(\PGL_3(\C))$ has an orbit of dimension 4 and the action of $\psi(\PGL_3(\C))$ does not.
\end{proof}

\subsection{A rational action on the space of plane conics}
Our goal is to extend the group homomorphism 
$\varphi\colon\PGL_3(\C)\to\PGL_6(\C)$
 to a group homomorphism 
 \[
 \Phi\colon\Cr_2\to\Cr_5.
 \]

A first naive idea is to check whether the map $\Psi\colon\{\PGL_3(\C), \sigma\}\to\Cr_5$ defined by $\Psi(g)=\varphi(g)$ for $g\in\PGL_3(\C)$ and 
$\Psi(\sigma)=[x_0^{-1}:x_2^{-1}:\dots:x_5^{-1}]$ extends to a group homomorphism $\Cr_2\to\Cr_5$. However, $\Psi(\sigma)$ and $\Psi(h)$ don't satisfy relation $(3)$ of Lemma \ref{relations}. Let $h=[Z-Z:Z-Y:Z]\in\Cr_2$, then 
\[
([x_0^{-1}:x_1^{-1}:\dots:x_5^{-1}]\circ\varphi(h))^3\neq \id.
\]

In \cite{MR1714823}, Gizatullin constructs an extension $\Phi\colon\Cr_2\to\Cr_5$ of $\varphi$, defined by $\Phi|_{\PGL_3(\C)}=\varphi$ and
\[
\Phi(\sigma)=[x_1x_2:x_0x_2:x_0x_1:x_3x_0:x_4x_1:x_5x_2].\]
He shows the following:

\begin{proposition}[\cite{MR1714823}]
The map $\Phi\colon \Cr_2\to \Cr_5$ is a group homomorphism.
\end{proposition}

\subsection{The dual action}\label{dualaction}
We can also look at the representation $\varphi^\vee\colon\PGL_3(\C)\to\PGL_6(\C)$ that is defined by 
\[
\varphi^\vee(g):=\prescript{t}{}\varphi(g)^{-1}.
\]
In other words, $\varphi^\vee=\varphi\circ\alpha$, where $\alpha\colon\PGL_3(\C)\to\PGL_3(\C)$ is the algebraic automorphism $g\mapsto ({}^tg)^{-1}$.

Let $A=(a_{ij})$ be a $3\times 3$ matrix. The cofactor matrix $C(A)$ of $A$ is given by 
\[
C_{ij}(A)=(-1)^{i+j}A_{ij},
\]
 where $A_{ij}$ is the $i,j$-minor of $A$, i.e. the determinant of the $2\times 2$-matrix obtained by removing the $i$-th row and $j$-th column of $A$. We denote by 
\[
Ad(A):=\prescript{t}{}C(A)
\] 
the {\it adjugate} matrix of $A$. This is a classical construction and it is well known that $Ad(AB)=Ad(B)Ad(A)$ and that if $A$ is invertible, then $Ad(A)=\det(A)A^{-1}$. In particular, $Ad\colon\p M_3\dashrightarrow\p M_3$ is a birational map. The conic corresponding to the symmetric matrix $A$ is the dual of the conic corresponding to the symmetric matrix $A$. This is one of the birational maps that  A.\,R.\,Williams described 1938 in his paper ``Birational transformations in 4-space and 5-space'' (\cite{MR1563724}). 

\begin{lemma}\label{ad}
We identify $\p^5$ with the projectivized space of symmetric $3\times 3$ matrices $\p M_3$. The birational transformation $Ad\in\Cr_5$ is given by 
\[
Ad:=[x_1x_2-x_3^2:x_0x_2-x_4^2:x_0x_1-x_5^2:x_4x_5-x_0x_3:x_3x_5-x_1x_4:x_3x_4-x_2x_5].
\]
Moreover, $ad$ conjugates $\varphi$ to $\varphi^\vee$.
\end{lemma}

\begin{proof}
It is a straightforward calculation that the rational map $Ad$ from $\p^5$ to itself that corresponds to $Ad$ is given by
\[
Ad:=[x_1x_2-x_3^2:x_0x_2-x_4^2:x_0x_1-x_5^2:x_4x_5-x_0x_3:x_3x_5-x_1x_4:x_3x_4-x_2x_5].
\]

The actions of $\PGL_3(\C)$ on $\p M_3$ induced by $\varphi$ and $\varphi^\vee$ are given by
$\varphi(g)(X)=gX(\prescript{t}{}g)$
and
$\varphi^\vee(g)X=\prescript{t}{}(g^{-1})Xg^{-1}$
respectively, for all $X\in \p M_3$. We obtain
\[
Ad(\varphi(g)(X))=Ad(\prescript{t}{}g)Ad(X)Ad(g)=(\prescript{t}{}g)^{-1}Ad(X)g^{-1}=\varphi^\vee(g)Ad(X).
\]
\end{proof}

\begin{remark}
The Blow-up $Q$ of $\p^5$ along the Veronese surface is the so called {\it space of complete conics}. Let $U\subset\p^5$ be the open orbit of the $\PGL_3$-action on $\p^5$ given by $\varphi$, i.e. $U=\p^5\setminus S$. Then $U$ can be embedded into $\p(\C^6)\times\p((\C^6)^\vee)$ by sending a conic $C\in U$ to the pair $(C, C^\vee)$, where $\C^\vee$ denotes the dual conic of $C$. It turns out that $Q$ is isomorphic to the closure of $U$ in $\p(\C^6)\times\p((\C^6)^\vee)$. Moreover, the $\PGL_3$-action on $\p^5$ given by $\varphi$ lifts to an algebraic action on $Q$ and the birational map $ad$ to an automorphism of $Q$. More details on this subject can be found for example in \cite{brion1989spherical}.
\end{remark}

Lemma \ref{ad} shows that the representations $\varphi$ and $\varphi^\vee$ are conjugate to each other in $\Cr_5$ by the birational transformation $Ad$. 
By conjugating $\Phi(\sigma)$ with $Ad$ we can extend $\varphi^\vee$ to the {\it dual embedding} $\Phi^\vee\colon\Cr_2\to\Cr_5$ and obtain
\[
\Phi^\vee(\sigma)=[(x_1x_2-x_3^2)^2x_0: (x_0x_2-x_4^2)^2x_1:(x_0x_1-x_5^2)^2x_2:  \]\[ (x_0x_2-x_4^2)(x_0x_1-x_5^2)x_3: (x_1x_2-x_3^2)(x_0x_1-x_5^2)x_4: (x_1x_2-x_3^2)(x_0x_2-x_4^2)x_5].
\]

\subsection{Geometry of $\Phi$}\label{geometry}
The embedding $\Phi$ induces a rational action of $\Cr_2$ on the space of conics on $\p^2$. The action of $\Phi(\sigma)$ can be viewed geometrically as follows (compare with \cite[Introduction]{MR1714823}): Let $Q_0:=[1:0:0]$, $Q_1:=[0:1:0]$ and $Q_2:=[0:0:1]$. Let $C\subset\p^2$ be a conic that doesn't pass through any of the points $Q_i$. Write
\[
C=\{a_{00}X^2+a_{11}Y^2+a_{22}Z^2+2a_{12}YZ+2a_{02}XZ+2a_{01}XY=0\}\subset\p^2
\]

Denote by $P_{1i}, P_{2i}$ the points of intersection of $C$ with the lines $l_i$, where $l_0:=\{X=0\}$, $l_1:=\{Y=0\}$ and $l_2:=\{Z=0\}$.  Denote by $f_{1i}$ and $f_{2i}$ the lines passing through $Q_i$ and $P_{i1}$ respectively through $Q_i$ and $P_{i2}$. The images $\sigma(f_{ji})$ are again lines passing through the points $Q_i$. Let $P'_{1i}$ and $P'_{2i}$ be the intersection points of $\sigma(f_{1i})$ and $\sigma(f_{2i})$ with $l_i$. One checks that the conic $D$ defined by the equation
\[
a_{11}a_{22}x_0^2+a_{00}a_{22}x_1^2+a_{00}a_{11}x_2^2+2a_{00}a_{12}x_1x_2+2a_{11}a_{02}x_0x_2+2a_{22}a_{01}x_0x_2=0
\]
passes through the points $P_{ij}'$. Since no 4 of the 6 points $P_{ij}$ lie on the same line, $D$ is the unique conic through the points $P_{ij}$. We have thus proven the following:

\begin{proposition}
For a general conic $C\subset\p^2$ there exists a unique conic $D$ through the six points $P_{ij}^\prime$ and $D$ is the image of $C$ under $\Phi(\sigma)$.
\end{proposition}

Notice as well that the indeterminacy points of $\Phi(\sigma)$ in $\p^5$ correspond to the subspace of dimension 2 of conics passing through the points $Q_1, Q_2, Q_3$ and the subspaces of dimension 2 of conics consisting of one $l_i$ and any other line. The three subspaces of dimension 4 of conics passing through one of the points $Q_i$ are contracted by the action of $\Phi(\sigma)$ and form the exceptional divisor.

In homogeneous coordinates of $\p^5$, the four planes of indeterminacy locus of $\Phi(\sigma)$ can be described as follows
\[
E_0=\{x_1=x_2=x_3=0\}, E_1=\{x_0=x_2=x_4=0\}, E_2=\{x_0=x_1=x_5=0\}
\]
\[
\text{ and } F=\{x_1=x_2=x_3=0\}.
\]

The exceptional divisor of $\Phi(\sigma)$ consists of the three hyperplanes 
\[
H_0=\{x_0=0\}, H_1=\{x_1=0\}, H_2=\{x_2=0\},
\]
The hyperplanes $H_0, H_1$ and $H_2$ are contracted by $\Phi(\sigma)$ onto the planes $E_0, E_1$ and $E_2$ respectively. Note as well that $E_0, E_1$ and $E_2$ are contained in the secant variety $S\subset\p^5$ of the Veronese surface $V$ and they are tangent to $V$. 

The geometrical description of the rational action of $\Phi^\vee(\sigma)$ on the space of conics is the dual of the construction described above. If $C$ is a conic not passing through any of the points $Q_0, Q_1, Q_2$, we get $\Phi^\vee(\sigma)(C)$ in the following way: let $l_{i,1}, l_{i,2}$ be the tangents of $C$ passing through the point $Q_i$. Then the images of the $l_{i,1}$ and $l_{i,2}$ under $\sigma$ are lines again. There exists a unique conic having all the lines $\sigma(l_{i,1})$ and $\sigma(l_{i,2})$ for all $i$ as tangents. 

These geometrical constructions show that $\Phi(\Cr_2)$ preserves the space of conics consisting of double lines and therefore the Veronese surface $V$ in $\p^5$. The injective morphism
\[
v\colon \p^2\to\p^5, \hspace{1mm}[X:Y:Z]\mapsto [X^2:Y^2:Z^2:YZ:XZ:XY]
\]
is called the {\it Veronese morphism}. It is an isomorphism onto its image, which is $V$.
It is well known that $v$ is $\PGL_3(\C)$-equivariant with respect to the standard action and the action induced by $\Phi$ respectively. The restriction of $\Phi(\sigma)$ to $V$ is a birational transformation. We therefore obtain a rational action  of $\Cr_2$ on $V\simeq\p^2$. Since the restrction of this rational action to $\PGL_3(\C)$ is the standard action, we obtain by Corollary \ref{autgn} that $v$ is $\Cr_2$-equivariant.

We observe as well that $\Phi(\Cr_2)$ preserves the secant variety $S\subset\p^5$ of $V$. Note that $S$ is the image of the morphism:
\[
\s\colon\p^2\times\p^2\to S\subset\p^5, \]
that maps the point $[X:Y:Z],[U:V:W]\in\p^2\times\p^2$ to the point 
\[ [XU:YV:ZW:1/2(YW+UZ):1/2(XW+ZU):1/2(XV+YU)].
\]
Note that $s$ is generically $2:1$.
Again, the geometrical construction above shows that $s$ is $\Cr_2$-equivariant with respect to the diagonal action on $\p^2\times\p^2$ and the action given by $\Phi$ on $\p^5$ respectively.

We obtain the following sequence of $\Cr_2$-equivariant maps:
 
 \[
 \p^2\xrightarrow{\Delta} \p^2\times\p^2\xrightarrow{s}\p^5,
 \]
 where $\Delta$ is the diagonal embedding.
 This proves part (2) to (4) of Theorem \ref{gizatullintheorem}.

 The observation that $\Phi(\Cr_2)$ preserves the Veronese surface and extends the canonical rational action of $\Cr_2$ has a nice consequence:

\begin{proposition}\label{degreegeq}
Let $f\in\Cr_2$. Then $\deg(f)\leq\deg(\Phi(f))$.
\end{proposition}

\begin{proof}
Denote by $v\colon\p^2\to\p^5$ the Veronese embedding. Let $C\subset\p^2$ be a general conic. The image $v(C)\subset \p^5$ is a curve of degree 4 given by the intersection of a hyperplane $H\subset \p^5$ and the Veronese surface. Let $f\in\Cr_2$ be a birational transformation of degree $d$. The strict transform $f(C')$ of a general conic $C'\subset\p^2$ intersects $C$ in $4d$ different points. So $v(C)$ intersects $v(f(C'))$ in $4d$ different points. By the above results we know that $v(f(C'))=\Phi(f)(v(C'))$. Let $d'=\deg(\Phi(f))$. Since $v(C')$ is a curve of degree 4, this yields that $\Phi(f)(v(C'))$ is a curve of degree $4d'$. The curve $\Phi(f)(v(C'))$ intersects the hyperplane $H$ in $4d$ points, hence \hbox{$d'\geq d$.}
\end{proof}

 \subsection{Two induced embeddings from $\Cr_2$ into $\Cr_4$}\label{inducedsection}
The birational map $Ad\in\Cr_5$ contracts the secant variety $S\subset\p^5$ onto the Veronese surface $V\subset\p^5$.  However, the exceptional locus of $\Psi^\vee(\sigma)=Ad\Phi(\sigma)Ad$ consists of the three hyperplanes 
\[
G_0=\{z_1z_2-z_3^2=0\}, G_1=\{z_0z_2-z_4^2=0\},G_2=\{z_0z_1-z_5^2=0\},
\]
with respect to homogeneous coordinates $[z_0:z_1:z_2:z_3:z_4:z_5]$ of $\p^5$.

This implies in particular that the restriction of $\Phi^\vee(\sigma)$ to $S$ induces a birational map of~$S$ and therefore that any element in $\Phi^\vee(\Cr_2)$ restricts to a birational map of $S$.
 
Since $S$ is a cubic hypersurface and contains the two disjoint planes 
\[
E_1=\{z_1=z_2=z_3=0\}, E_2=\{z_0=z_4=z_5=0\},
\]
it is rational. Explicitely, projection onto $E_1$ and $E_2$ yields the birational map $A\colon S\dashrightarrow  \p^2\times\p^2$ defined by
\[
[z_0:z_1:z_2:z_3:z_4:z_5]\mapsto [z_1:z_2:z_3],[z_0:z_4:z_5].
\]
The inverse transformation $A^{-1}$ is given by
\[
[x_0:x_1:x_2],[y_0:y_1:y_2]\mapsto [p_2y_0, p_1x_0, p_1x_1, p_1x_2, p_2y_1, p_2y_2],
\]
where 
$p_1=(x_0y_1^2+x_1y_2^2-2x_2y_1y_2) \text{  and } p_2=y_0(x_0x_1-x_2^2).$

Let $f\in\Cr_2$. As seen above, both images $\Phi(f)$ and $\Phi^\vee(f)$ restrict to a birational map of $S$. So conjugation of $\Phi$ and $\Phi^\vee$ by $A$ yields two embeddings from $\Cr_2$ into $\Bir(\p^2\times\p^2)\simeq\Cr_4$, which we  denote by $\Psi_1$ and $\Psi_2$ respectively.

\begin{proof}[Proof of Proposition $\ref{inducedembedd}$]
Irreducibility is proved in Section \ref{irredsection}.

By Theorem \ref{poptori}, all tori $D_2\subset\Cr_4$ are conjugate to the standard torus $D_2\subset \Cr_4$. We calculate the map that conjugates $\Psi_1(D_2)=\Psi_2(D_2)$ to the image of the standard embedding of $D_2$ explicitely.
Let $\rho\colon\p^2\times\p^2\dashrightarrow\p^2\times\p^2$
be the birational transformation defined by
\[
([x_0:x_1:x_2],[y_0:y_1:y_2])\mapsto([x_2y_0:x_0y_1:x_2y_1], [x_0y_1^2:x_1y_2^2:x_2y_1y_2]).
\]
The inverse map $\rho^{-1}$ is given by 
\[
([x_0:x_1:x_2],[y_0:y_1:y_2])\mapsto ([x_1^2y_2^2:x_2^2y_0y_1:x_1x_2y_2^2], [x_0y_0:x_2y_0:x_1y_2]).
\]
One calculates that $\rho A\Psi_1([aX:bY:cZ]) A^{-1}\rho^{-1}$
maps $([x_0:x_1:x_2], [y_0:y_1:y_2])$ to $([ax_0:bx_1:cx_2], [y_0:y_1:y_2])$. Correspondingly, 
$\rho A\Psi_2([aX:bY:cZ]) A^{-1}\rho^{-1}$
maps $([x_0:x_1:x_2], [y_0:y_1:y_2])$ to $([a^{-1}x_0:b^{-1}x_1:c^{-1}x_2], [y_0:y_1:y_2])$. So the second coordinates parametrize the closures of the $D_2$-orbits. Since $\W_2$ normalizes $D_2$, its image preserves the $D_2$-orbits. We thus obtain two homomorphisms
\[
\chi_1\colon\W_2\to\Cr_2,  \chi_2\colon\W_2\to\Cr_2
\]
by just considering the rational action of $\W_2$ on the second coordinate. 

Assume that there exists an element $A\in\Bir(\p^2\times\p^2)$ that conjugates $\Psi$ to $\Psi^\vee$. As $A$ normalizes $\Psi_1(D_2)=\Psi_2(D_2)$, it preserves the $\Psi_1(D_2)$-orbits as well. Hence by restriction on the second coordinate, it conjugates $\chi_1$ to $\chi_2$. It therefore suffices to show that $\chi_1$ and $\chi_2$ are not conjugate.

In $\Cr_2$ we have
\[
f:=[XY:YZ:Z^2]=\tau_1g_0\sigma g_0\sigma g_0\tau_2,
\]
where $\tau_1=[Z:Y:X], \tau_2=[Y:Z:X]$ and $g_0=[Y-X:Y:Z]$. By calculating the corresponding images under $\Phi$ we obtain
\[
\Phi(f)=\Phi(\tau_1g_0\sigma g_0\sigma g_0\tau_2)=[x_0x_1:x_1x_2:x_2^2: x_2x_3: -x_2x_5+2x_3x_4: x_1x_4]
\]
and $\Phi^\vee(f)=[g_0:g_1:g_2:g_3:g_4:g_5]$,
where
\[
g_0=(x_0 x_1-x_5^2)^2 x_0,
\]
\[
g_1=x_0^2 x_1^2 x_2-2 x_0 x_1 x_2 x_5^2-4 x_0 x_1 x_3 x_4 x_5+4 x_0 x_3^2 x_5^2+4 x_1 x_4^2 x_5^2+x_2 x_5^4-4 x_3 x_4 x_5^3,
\]
\[
g_2=(x_0 x_2-x_4^2)^2 x_1,
\]
\[
g_3=(x_0 x_2-x_4^2) (x_0 x_1 x_3-2 x_1 x_4 x_5+x_3 x_5^2),
\]
\[
g_4= -(x_0 x_2-x_4^2) (x_0 x_1-x_5^2) x_5,
\]
\[
g_5=(x_0 x_1-x_5^2) (x_0 x_1 x_4-2 x_0 x_3 x_5+x_4 x_5^2).
\]
This yields
\[
\chi_1(f)=[(y_1-2y_2)^2: y_0y_1:-y_2(y_1-2y_2)]
\]
and
\[
\chi_2(f)=[y_0^2 y_1+4 y_0 y_1^2-6 y_0 y_1 y_2-3 y_1 y_2^2+4 y_2^3 : y_0 (y_0+2 y_1-3 y_2)^2 :
\]
\[ (2 y_0 y_1-y_0 y_2-y_2^2) (y_0+2 y_1-3 y_2)].
\]
We show that these two transformations are not conjugate in $\Cr_2$.
With respect to affine coordinates $[y_0:y_1:1]$ one calculates
\[
\chi_1(f)^2=\left(\frac{y_0y_1-2y_1+4}{y_1-2}, y_1 \right).
\]
From this we see that the integer sequence $\deg(\chi_1(f)^n)$ grows linearly in $n$ and is, in particular, not bounded. 

Let $A=[y_0-y_2:y_1-y_2:y_2].$ Then 
\[
A\chi_2(f)^2A^{-1}=[-y_0^2 y_1^2 (2 y_1+y_0): y_0^2 y_1^2 (3 y_1+2 y_0): p(y_0, y_1, y_2 )(3 y_1+2 y_0)(2 y_1+y_0) ],
\]
where $p(y_0, y_1, y_2)=(6 y_1^2 y_2+7 y_2 y_0 y_1+6 y_0 y_1^2+2 y_0^2 y_2+2 y_0^2 y_1)$.
We claim that 
\[
f_A^n=A\chi_2(f)^{2n}A^{-1}=[-y_0^2y_1^2(2ny_1+(2n-1)y_0):y_0^2y_1^2((2n+1)y_1+2ny_0):f_n],
\]
where $f_n=(2ny_1+(2n-1)y_0)((2n+1)y_1+2ny_0)p_n(y_0, y_1, y_2)$ for some homogeneous $p_n\in\C[y_0, y_1, y_2]$ of degree $3$. Note that this claim implies in particular that $\deg(\chi_2(f)^n)$ is bounded for all $n$ and hence that $\chi_1(f)$ and $\chi_2(f)$ are not conjugate.

To prove the claim we proceed by induction. Assume that $f_A^n$ has the desired form. One calculates that the first coordinate of $f_A^{n+1}=A\chi_2(f)^{2}A^{-1}\circ f_A^n$ is 
\[
-ry_0^2y_1^2((2n+2)y_1+(2n+1)y_0),
\]
the second coordinate is
\[
ry_0^2y_1^2((2n+3)y_1+(2n+1)y_0)
\]
and the third coordinate
\[
r((2n+2)y_1+(2n+1)y_0)((2n+3)y_1+(2n+1)y_0)p_{n+1}(x_0, x_1, x_2),
\]
where $r=y_0^4y_1^4(2ny_0+(2n-1)y_1)^2((2n+1)y_0+2ny_1)^2$ and $p_n\in\C[x_0,x_1,x_2]$ is homogeneous of degree $3$. This proves the claim.
\end{proof}

 \subsection{A volume form}\label{volume}
Let $M$ be a complex projective manifold. It is sometimes interesting to study subgroups of $\Bir(M)$ that preserve a given form. In \cite{MR3080816} and \cite{MR3451389} the authors study for example birational maps of surfaces that preserve a meromorphic symplectic form (see \cite{Corti:2015aa} for the 3-dimensional case). In \cite{MR2433658} and \cite{Cerveau:2016aa} Cremona transformations in dimension 3 preserving a contact form are studied.

Define 
\[
F:=\det\left(
\begin{array}{ccc}
 x_0&x_5&x_4    \\
  x_5& x_1 &x_3  \\
  x_4&x_3&x_2\\ 
\end{array}
\right)
\]
and let 
\[
\Omega:=\frac{x_5^6}{F^2} \cdot dx_0\wedge dx_1\wedge dx_2\wedge dx_3\wedge dx_4.
\]
Then $\Omega$ is a $5$-form on $\p^5$ with a double pole along the secant variety of the Veronese surface. Note that the total volume of $\p^5$ is infinite.

\begin{proposition}
All elements in $\Phi(\Cr_2)$ preserve $\Omega$.
\end{proposition}

\begin{proof}
We show that $\Phi(\PGL_3(\C))$ and $\Phi(\sigma)$ preserve $\Omega$.

Let $g=[-X:-Y:Z]\in\Phi(\PGL_3(\C))$. One checks that $\Phi(g)$ preserves $\Omega$. Since $\Phi(\PGL_3(\C))$ preserves $F$, we have that $\Phi(fgf^{-1})$ preserves $\Omega$ as well. As $\Phi(\PGL_3(\C))$ is simple, the whole group preserves $\Omega$.

With respect to affine coordinates given by $x_5=1$, we have
\[
\Phi(\sigma)=(x_1, x_0, x_0x_1x_2^{-1},x_0x_3x_2^{-1}, x_1x_4x_2^{-1}).
\]
A direct calculation yields $\Omega\circ\Phi(\sigma)=\Omega$.
\end{proof}

\subsection{Polynomial automorphisms}\label{polauto}
In this section we will prove Claim (6) of Theorem \ref{gizatullintheorem} as well as Theorem \ref{degreeaut}. Let $\aut(\A^2)\subset\Cr_2$ be the subgroup of automorphisms of the affine plane with respect to the affine coordinates $[1:X:Y]$. By the theorem of Jung and van der Kulk (see for example \cite{MR1955604}), $\aut(\A^2)$ has the following amalgamated product structure
\[
\aut(\A^2)=\aff_2*_{\cap}\J_2,
\]
where $\J_2$ denotes the subgroup of elementary automorphisms, which is the subgroup of all elements of the form
\[
\left\{ (c_1X, c_2Y+p(X))\mid c_1, c_2\in\C, p(X)\in\C[X]\right\}.
\]
Let $f\in\aut(\A^2)$ and assume that $f=a_1j_1a_2j_2\cdots j_{n-1}a_n$, where $a_i\in\aff_2$ and $j_i\in\J_2\setminus\aff_2$. It is well known that $\deg(f)=\deg(j_1)\deg(j_2)\cdots\deg(j_{n-1})$. 

Let $\aut(\A^5)\subset\Cr_5$ be given by the affine coordinates $[1:x_1:\dots:x_5]$. Lemma~\ref{afftoaff} follows from a direct calculation.

\begin{lemma}\label{afftoaff}
The image $\Phi(\aff_2)$ is contained in $\aff_5$.
\end{lemma}
We consider the following elements in $\J_2$:
 \[
 f_n^\lambda:=(X, Y+\lambda X^n),
\]
where $n\in\Z_{\geq 0}$ and $\lambda\in\C$.

\begin{lemma}\label{jonqim}
For all $n\in\Z_{\geq 0}$ we have
\[
\Phi(f^\lambda_n)=(x_1, x_2+\lambda^2x_1^n+\lambda x_3A_n-\lambda x_4x_1A_{n-1}, x_3+\lambda x_1B_{n-1}, x_4+\lambda B_n, x_5),
\]
where 
\[
A_n=2\sum_{k=0}^{\floor*{n/2}}\binom{n}{2k+1}x_5^{n-2k-1}(x_5^2-x_1)^k
\]
and
\[
B_n=\sum_{k=0}^{\floor*{n/2}}\binom{n}{2k}x_5^{n-2k}(x_5^2-x_1)^k.
\]
Moreover, the following recursive identities hold:
\[
A_n=2x_5A_{n-1}-x_1A_{n-2},
\]
\[
B_n=2x_5B_{n-1}-x_1B_{n-2}.
\]
\end{lemma} 

\begin{proof}
For $n=0$ and $n=1$ the claim follows from a direct calculation.

Let $s:=(X, XY)\in\Cr_2$. Then we have $f^\lambda_{n+1}=sf^\lambda_ns^{-1}$. In $\Cr_2$ the identity
$s=\tau_1g_0\sigma g_0\sigma g_0\tau_2$ holds,
where $\tau_1=(XY^{-1},Y^{-1}), \tau_2=(Y^{-1},XY^{-1})$ and $g_0=(X, X_Y)$. Note that $\tau_1$ and $\tau_2$ are elements of $\PGL_3$. If we calculate the corresponding images under $\Phi$ we obtain
\[
\Phi(s)=\Phi(\tau_1g_0\sigma g_0\sigma g_0\tau_2)=(x_1, x_1x_2, x_1x_4, 2x_4x_5-x_3, x_5)
\]
and
\[
\Phi(s^{-1})=(x_1, x_2x_1^{-1}, 2x_3x_5x_1^{-1}-x_4, x_3x_1^{-1}, x_5).
\]
One calculates
\[
sf_n^\lambda s^{-1}=(x_1, x_2+\lambda^2x_1^{n+1}+\lambda x_3(2x_5-x_1)A_{n-1}-\lambda x_4x_1A_n, x_3+\lambda x_1B_n,
\]
\[
x_4-\lambda(2x_5B_n-x_1B_{n-1}).
\]
This shows by induction that
\[
\Phi(f^\lambda_n)=(x_1, x_2+\lambda^2x_1^n+\lambda x_3A_n-\lambda x_4x_1A_{n-1}, x_3+\lambda x_1B_{n-1}, x_4+\lambda B_n, x_5),
\]
where
\[
A_n=2x_5A_{n-1}-x_1A_{n-2}, A_0=0, A_1=2;
\]
\[
B_n=2x_5B_{n-1}-x_1B_{n-2}, B_0=1, B_1=x_5.
\]
These recursive formulas have the following closed form:
\[
A_n=\frac{ \left(  x_5+\sqrt {{{ x_5}}^{2}-{ x_1}}
 \right) ^{n}- \left( { x_5}-\sqrt {{{ x_5}}^{2}-{ x_1}}
 \right) ^{n}}{\sqrt {{{ x_5}}^{2}-{ x_1}}},
\]
\[
B_n=\frac{1}{2}\left( x_{{5}}-\sqrt {{x_{{5}}}^{2}-x_{{1}}} \right) ^{n}+1/2\,
 \left( x_{{5}}+\sqrt {{x_{{5}}}^{2}-x_{{1}}} \right) ^{n}.
\]
The claim follows.
\end{proof}

Since $\aff_n$ together with all the elements $f_n^\lambda$, $n\in\N$, $\lambda\neq 0$ generates $\aut(\A^2)$, Lemma \ref{jonqim} shows that $\Phi(\aut(\A^2))$ is contained in $\aut(\A^5)$ and thus claim (6) of Theorem \ref{gizatullintheorem}.

\begin{lemma}\label{AnBm}
Let $n$ and $m$ be positive integers and $A_n, B_m$ as in Lemma~$\ref{jonqim}$. Then
\[
A_nB_{m-1}-A_{n-1}B_m=P(x_1,x_5),
\]
where $P\in\C[x_1,x_5]$ is a polynomial of degree $<\max\{m,n\}$.
\end{lemma}

\begin{proof}
If $n=1$ or $m=1$ the claim is true, since $A_0=0, A_1=2, B_0=1,B_1=x_5$ and $\deg(A_k)=k-1, \deg(B_k)=k$. By the identities from Lemma \ref{jonqim}, one obtains
\begin{align*}
A_nB_{m-1}-A_{n-1}B_m& =(2x_5A_{n-1}-x_1A_{n-2})B_{m-1}-A_{n-1}(2x_5B_{m-1}-x_1B_{m-2})\\
&=x_1(A_{n-1}B_{m-2}-A_{n-2}B_{m-1}).
\end{align*}
The claim follows by induction on $m$ and $n$.
\end{proof}

\begin{lemma}\label{imj}
Let
\[
f=f_{1}^{\lambda_1}f_{2}^{\lambda_2}\cdots f_{n}^{\lambda_n}, \text{ where } \lambda_n\neq 0.
\]
Then 
\[
\Phi(f)=(x_1, x_2+F, x_3+p_3(x_1, x_5)+\lambda_n x_1B_{n-1}, 
x_4+p_4(x_1, x_5)+\lambda_n B_n, x_5),
\]
where $F=p_2(x_1,x_5)+x_3(\lambda_1 A_1+\dots+\lambda_n A_n)-x_4x_1(\lambda_1A_{n-1}+\cdots+\lambda_nA_n)$ and $p_2, p_3, p_4\in\C[x_1,x_5]$ are polynomials of degree $\leq n$. In particular, $\deg(\Phi(f))=\deg(f)$.
\end{lemma}

\begin{proof}
It is easy to see that the third and fourth coordinate of $\Phi(f)$ have the claimed form. The more difficult part is the second coordinate.

For $n=1$ the claim follows directly from Lemma \ref{jonqim}. We proceed now by induction. Let $\lambda_{n+1}\neq 0$ and $m$ be the largest number, such that $m\leq n$ and $n\neq 0$. By the induction hypothesis we may assume that the second coordinate of $\Phi(f_{1}^{\lambda_1}f_{2}^{\lambda_2}\cdots f_{m}^{\lambda_m})$ has the form
\[
x_2+p_2(x_1,x_5)+x_3(\lambda_1 A_1+\dots+\lambda_m A_m)-x_4x_1(\lambda_1A_{0}+\cdots+\lambda_mA_{m-1}).
\]
The second coordinate of $\Phi(f_{1}^{\lambda_1}f_{2}^{\lambda_2}\cdots f_{m}^{\lambda_m})\circ\Phi(f_n^{\lambda_n})$ is therefore
\[
x_2+p_2(x_1,x_5)+x_3(\lambda_1 A_1+\dots+\lambda_m A_m+\lambda_nA_n)-x_4x_1(\lambda_1A_{0}+\cdots+\lambda_mA_{m-1}+\lambda_nA_n)+
\]\[
x_1\sum_{k=1}^m\lambda_k(A_kB_{n-1}-A_{k-1}B_n).
\]
By Lemma \ref{AnBm}, $x_1\sum_{k=1}^m\lambda_k(A_kB_{n-1}-A_{k-1}B_n)$ is a polynomial in $x_1$ and $x_5$ of degree $\leq n$.
\end{proof}

\begin{proof}[Proof of Theorem $\ref{degreeaut}$]
The first claim was proved in Proposition \ref{degreegeq}.

For the second part it suffices by the remark above on the amalgamated product structure to show that $\deg(\Phi(f))=\deg(f)$ for all elements $f\in\J_2$. Composition with an element in $\aff_2$ doesn't change the degree. So it is enough to consider elements in $\J_2$ of the form 
$f=(X, Y+P(X)), P\in\C[X]$.
For suitable $\lambda_i\in\C$ we have
$f=f_{1}^{\lambda_1}f_{2}^{\lambda_2}\cdots f_{n}^{\lambda_n}, \text{ where } \lambda_n\neq 0.$
In Lemma \ref{imj} we've seen that $\Phi$ preserves the degree of these elements.
\end{proof}

\subsection{Irreducibility of $\Phi$, $\Psi_1$ and $\Psi_2$}\label{irredsection}

First we show that $\Phi$ is irreducible. Assume that there is a rational dominant map $\pi\colon\p^5\dashrightarrow M$ to a variety $M$ with an algebraic embedding $\varphi_M\colon \Cr_2\to\Bir(M)$ such that $A$ is $\Cr_2$-equivariant. Since $\varphi_M$ is algebraic, we may assume that $\PGL_3(\C)$ acts regularly on $M$. We obtain that the restriction of $A$ to the open $\PGL_3(\C)$-invariant subset $U\subset\p^5$ consisting of all smooth conics is a $\PGL_3(\C)$-equivariant morphism, whose image is an open dense subset of $M$ on which $\PGL_3(\C)$ acts transitively. Note that this implies $\dim(M)>1$. 

If $\dim M=2$, we obtain by Theorem \ref{thmserge} that $M\simeq \p^2$ with the standard action of $\PGL_3(\C)$.  The stabilizer  in $\PGL_3(\C)$ of a point in $U\subset\p^5$ is isomorphic to $\SO_3(\C)$. On the other hand the stabilizer in $\PGL_3(\C)$ of a point in $\p^2$ is isomorphic to the group of affine transformations $\aff_2=\GL_2(\C)\ltimes \C^2$. Since $\SO_3(\C)$ can not be embedded into $\aff_2$, the case $\dim(M)=2$ is not possible. 

If $\dim(M)=3$, we find, by Theorem \ref{pglmain}, a $\PGL_3(\C)$-equivariant projection $M\dashrightarrow\p^2$ and are again in the case $\dim(M)=2$.

If $\dim(M)=4$, let $p\in M$ be a general point and $F_p:=A^{-1}(p)\subset \p^5$ the fiber of $A$. Let $q\in F_p$ be a point that is only contained in one connected component $C$ of $F_p$. Again, the stabilizer of $q$ is isomorphic to $\SO_3(\C)$. This implies that $\SO_3(\C)$ acts regularly on the curve $C$ with a fixpoint. The group of birational transformations of $C$ is isomorphic to $\PGL_2(\C)$, is abelian or is finite. In all cases we obtain that the connected component of the identity $\SO_3(\C)^0$ fixes $C$ pointwise. In other words, the group $\SO_3(\C)^0$ preserves each conic of the family of conics in $\p^2$ parametrized by $C$. This is not possible.

The proof that $\Psi_1$ and $\Psi_2$ are irreducible is done analogously.


\section{$\PGL_{n+1}(\C)$-actions in codimension 1}\label{pglnactions}
In this section we look at algebraic embeddings of $\PGL_{n+1}(\C)$ into $\Bir(M)$ for complex projective varieties $M$ of dimension $n+1$. Our aim is to prove Theorem~\ref{pglmain}.

\begin{theorem}\label{pglmain}
Let $n\geq 2$ and let $M$ be a smooth projective variety of dimension $n+1$ with a non-trivial $\PGL_{n+1}(\C)$-action. Then, up to birational conjugation and automorphisms of $\PGL_{n+1}(\C)$, we have one of the following:
\begin{enumerate}
\item $M\simeq F_l=\p(\mathcal{O}_{\p^n}\oplus\mathcal{O}_{\p^n}(-l(n+1))$ for a unique element $l\in\Z_{\geq 0}$ and $\PGL_{n+1}(\C)$ acts as in Example $\ref{defl}$. 
\item $M\simeq\p^n\times C$ for a unique smooth curve $C$ and $\PGL_{n+1}(\C)$ acts on the first factor as in Example $\ref{product}$.
\item $M\simeq\p(T\p^2)$ and $\PGL_{3}$ acts as in Example $\ref{defb}$.
\item $M\simeq\mathbb{G}(1,3)$ and $\PGL_4(\C)$ acts as in Example $\ref{grass}$.
\end{enumerate}
Moreover, these actions are not birationally conjugate to each other.
\end{theorem}

\begin{remark}
If $M$ is rational and of dimension 2 or 3, this result can be deduced directly from the classification of maximal algebraic subgroups of $\Cr_n$ by Enriques, Umemura and Blanc (\cite{enriques1893sui}, \cite{MR683251}, \cite{MR803342}, \cite{MR676586}, \cite{MR2504924}).
\end{remark}

\subsection{Classification of varieties and groups of automorphisms}
With some geometric invariant theory and using results of Freudenthal about topological ends, the following classification can be made (see \cite{MR3014483} and the references in there):

\begin{theorem}\label{pglvar}
Let $M$ be a smooth projective variety of dimension $n+1$ with an action of $\PGL_{n+1}(\C)$, where $n\geq 2$. Then we are in one of the following cases:
\begin{enumerate}
\item $M\simeq\p(\mathcal{O}_{\p^n}\oplus\mathcal{O}_{\p^n}(-k))$ for some $k\in\Z_{\geq 0}$. 
\item $M\simeq\p^n\times C$ for a curve $C$ of genus $\geq 1$.
\item $M\simeq\p(T\p^2)\simeq\PGL_3(\C)/ B$, where $P\subset\PGL_3(\C)$ is a maximal Borel subgroup.
\item $M\simeq\mathbb{G}(1,3)\simeq\PGL_4(\C)/P$, where $P\subset\PGL_4(\C)$ is the parabolic subgroup consisting of matrices of the form
\[
\left( \begin{array}{cccc} 
* & *&*&* \\
*& *&*&*\\
0&0&*&*\\
0&0&*&*\\
 \end{array}\right).
\]
\end{enumerate}
\end{theorem}
The connected components $\aut^0(M)$ of the automorphism groups of the varieties $M$ that appear in Theorem \ref{pglvar} are well known. Proofs of the following Proposition can be found in \cite{MR1334091}.

\begin{proposition}\label{autogroup}
We have
\begin{itemize}
\item 
$\aut^0(\p(\mathcal{O}_{\p^n}\oplus\mathcal{O}_{\p^n}(-k))\simeq (\GL_{n+1}(\C)/\mu_k)\rtimes\C [x_0,\dots,x_n]_k$, where \newline $\C [x_0,\dots,x_n]_k$ denotes the additive group of homogeneous polynomials of degree $k$ and $\mu_k\subset\C^*$ the group of all elements $c\in\C^*$ satisfying $c^k=1$,
\item $\aut^0(\p^n\times C)\simeq \PGL_{n+1}(\C)\times \aut^0(C)$,
\item $\aut^0(\p(T\p^2))\simeq \PGL_3(\C)$,
\item $\aut^0(\mathbb{G}(1,3))\simeq \PGL_4(\C)$.
\end{itemize}
\end{proposition}
To describe the $\PGL_{n+1}(\C)$-actions on these varieties we recall some results about group cohomology.

\subsection{Group cohomology} Let $H$ be a group that acts by automorphisms on a group $N$. A {\it cocycle} is a map $\tau\colon H\to N$ such that $\tau(gh)=\tau(g)(g\cdot\tau(h))$ for all $g,h\in H$.
Two cocycles $\tau$ and $\nu$ are {\it cohomologous} if there exists an $a\in N$ such that 
\[
\tau(g)=a^{-1}\nu(g)(g\cdot a)\hspace{1mm} \text{ for all }g\in H.
\]
The set of cocycles up to cohomology will be denoted by $H^1(H,N)$. If $H$ acts trivially on $N$, the set $H^1(H,N)$ corresponds to the set of group homomorphisms $H\to N$. The following lemma is well known.

\begin{lemma}\label{semidirect}
Let $G:=N\rtimes H$ be a semi direct product of groups and $\pi\colon G\to H$ the canonical projection on $H$. Then there exists a bijection between $H^1(H,N)$ and the sections of $\pi$ up to conjugation in $N$.
\end{lemma}

There always exists the trivial cocycle $\tau_0\colon H\to N$, $g\mapsto e_N$. The set $H^1(G,N)$ is therefore a pointed set with basepoint $\tau_0$. Assume that $G$ acts on two groups $A$ and $B$ by automorphisms. A $G$-homomorphism $\phi\colon A\to B$ induces a homomorphism of pointed sets 
\[
\phi_*\colon H^1(G,A)\to H^1(G,B)
\] 
given by $\phi_*(\tau)=\phi\circ\tau.$

\begin{proposition}[\cite{MR554237}, p. 125, Proposition 1]\label{propserre}
Let $G$ be a group that acts by automorphisms on groups $A,B$ and $C$. Every exact sequence of $G$-homomorphisms
\[
1\to A\to B\to C\to 1
\]
induces an exact sequence of pointed sets
\[
H^1(G,A)\to H^1(G,B)\to H^1(G,C).
\]
\end{proposition}

\subsection{Proof of Theorem $\ref{pglmain}$}
\subsubsection{Uniqueness of the actions}
Now we show that $\PGL_{n+1}(\C)$ can only be embedded into $\aut^0(\p(\mathcal{O}_{\p^n}\oplus\mathcal{O}_{\p^n}(-k))$ if and only if $n\mid k$.  Then we show that in this case, up to conjugation and algebraic automorphisms of $\PGL_{n+1}(\C)$, the embedding is unique. 

By Proposition \ref{autogroup}, $\aut^0(\p(T\p^2))\simeq \PGL_3(\C)$ and $\aut( \mathbb{G}(1,3))\simeq\PGL_4(\C)$. The uniqueness of the embedding is clear in this cases since $\PGL_{n+1}(\C)$ is a simple group. If $M\simeq \p^n\times C$ uniqueness follows directly from the fact that $\PGL_{n+1}(\C)$ does not embed into $\aut(C)$. 

\begin{lemma}\label{existencesection}
A non-trivial group homomorphism $\PGL_n(\C)\to\GL_n(\C)/\mu_k$ exists if and only if $n\mid k$, where
\[
\mu_k=\{\lambda\id\mid \lambda\in\C, \lambda^k=1\}.
\]
\end{lemma}

Let $n$ and $k$ be positive integers such that $(n+1)\mid k$. Denote by $\C[x_0,\dots,x_n]_k$ the vector space of homogeneous polynomials of degree $k$ . We define
\[
G:=\C[x_0,\dots,x_n]_k\rtimes\PGL_{n+1}(\C),
\]
where the semi direct product is taken with respect to the action $g\cdot p=p\circ g^{-1}$. Here we look at $\PGL_{n+1}(\C)\subset\GL_{n+1}(\C)/\mu_k$ as described in Lemma \ref{existencesection}. Let  $
\pi\colon G\to\PGL_{n+1}(\C)$ be the standard projection.

\begin{lemma}\label{cohomv}
Up to conjugation, there exists a unique section
$\iota\colon \PGL_{n+1}(\C)\to G$ of $\pi$.
\end{lemma}
The results in Lemma \ref{existencesection} and Lemma \ref{cohomv} are certainly well known. A detailed proof can be found in \cite{longversion}.

\begin{lemma}
$\PGL_{n+1}(\C)$ acts non-trivially on the fibration $\p(\mathcal{O}_{\p^n}\oplus\mathcal{O}_{\p^n}(-k))$ with basis $\p^n$ if and only if $k=l(n+1)$ for some nonnegative $l$. Moreover, in this case the action is unique up to conjugation and up to algebraic automorphisms of $\PGL_{n+1}(\C)$.
\end{lemma}

\begin{proof}
Let $\phi\colon \PGL_{n+1}(\C)\to\aut^0(\p(\mathcal{O}_{\p^n}\oplus\mathcal{O}_{\p^n}(-k)))$ be an algebraic embedding. By Proposition \ref{autogroup}, there exists an exact sequence of algebraic homomorphisms
\[
1\to\C[x_0,\dots,x_n]_k\to\aut^0(\p(\mathcal{O}_{\p^n}\oplus\mathcal{O}_{\p^n}(-k)))\to\GL_{n+1}(\C)/\mu_k\to 1.
\]
If $\phi$ is non-trivial, this induces a non-trivial algebraic homomorphism from \newline $\PGL_{n+1}(\C)$ into $\GL_{n+1}(\C)/\mu_k$ and by Lemma \ref{existencesection} this is possible if and only if $(n+1)\mid k$.
So assume that $k=l(n+1)$ for an integer $l$. It remains to show that in this case $\phi$ is unique up to conjugation and up to algebraic automorphisms of $\PGL_{n+1}(\C)$. Let 
\[
F_l:=\p(\mathcal{O}_{\p^n}\oplus\mathcal{O}_{\p^n}(-k)).
\]
We look at $F_l$ as a $\p^1$-fibration over the basis $\p^{n}$. So there is an exact sequence

\[
1\to \aut^0_{\p^n}(F_l)\to \aut^0(F_l)\xrightarrow{\pi}\PGL_n(\C)\to 1.
\]
Here, $\aut^0_{\p^n}(F_l)\simeq\C^*\ltimes\C[x_0,\dots, x_n]_k$ denotes the subgroup of automorphisms of $F_l$ that fix the basis $\p^n$ pointwise. 

Let $H:=\PGL_{n+1}(\C)$.  By Lemma \ref{semidirect}, the sections of $\pi$ up to conjugation are in bijection with 
\[
H^1(H, \aut^0_{\p^n}(F_l))=H^1(H, \C[x_0,\ldots,x_n]_k\rtimes\C^*/\mu_k).
\] 
By Proposition \ref{propserre}, there is an exact sequence of pointed sets 
\[
H^1(H, \C[x_0,\dots,x_n]_k)\to H^1(H, \aut_{\p^n}(F_l))\to H^1(H, \C^*/\mu_k).
\]
The action of $H$ on $\C^*/\mu_k$ is trivial, so $H^1(H,\C^*/\mu_k)$ is the set of homomorphisms $H\to\C^*/\mu_k$. Hence $ H^1(H, \C^*/\mu_m)=\{1\}$. By Lemma \ref{cohomv}, we obtain $H^1(H, \C[x_0,\dots,x_n]_k)=\{1\}$ and thus $H^1(H, \aut_{\p^n}(F_l))=\{1\}$.
So, all sections of $\pi$ are conjugate.

Now, since $H$ is simple and not contained in $\aut^0_{\p^n}(F_l)$, we obtain
$\pi\circ\phi(H)\subset H.$
Both $\phi$ and $\pi$ are are algebraic morphisms, so 
$\pi\circ\phi(H)=H.$
Therefore, up to the algebraic automorphism $\pi\circ\phi$, the homomorphism $\phi$ is a section of $\pi$.
\end{proof}

\subsubsection{Non conjugacy}
It remains to show that the actions from Theorem \ref{pglmain} are not birationally conjugate. 

Let $M$ be a variety of dimension $n+1$ on which $\PGL_{n+1}(\C)$ acts faithfully.

If $M$ is not rational, then $M$ is isomorphic to $\p^n\times C$ for some smooth curve $C$. Recall that $\p^n\times C$ is birationally equivalent to $\p^n\times C'$ for smooth curves $C$ and $C'$ if and only if $C$ and $C'$ are birationally equivalent which again implies that $C$ and $C'$ are isomorphic. So, if $\PGL_{n+1}(\C)$ acts rationally and non trivially on a non rational variety $M$ of dimension $n+1$, then this one is uniquely determined up to algebraic automorphisms of $\PGL_{n+1}(\C)$ and up to birational conjugation in $\Bir(M)$.

In the case that $M$ is rational, we have to show that the $\PGL_{n+1}(\C)$-actions listed in Theorem \ref{pglmain} are not conjugate to each other. For this, note that none of them has an orbit of codimension $\geq 1$.  Lemma \ref{exci} induces therefore that any birational transformation conjugating one action to another one must be an isomorphism. As the varieties listed in Theorem \ref{pglmain} are not isomorphic we conclude that the actions are not conjugate.


\section{Extension to $\Cr_n$}\label{extension}
In this section we study how the $\PGL_{n+1}(\C)$-actions described in the above section extend to rational $\Cr_n$-actions. Our goal is to prove Theorem \ref{crmain}. We proceed case by case.

\subsection{The case $\mathbb{G}(1,3)$}
Let

\[s_1\coloneqq \left( \begin{array}{ccc}
0 & 0 &1\\
1 & 0 &0\\
0& 1& 0\\
 \end{array} \right),
  \text{ and } 
 s_2\coloneqq \left( \begin{array}{ccc}
0 & -1 &1\\
0& -1 &0\\
1& -1& 0\\
 \end{array} \right)\in\GL_3(\Z).
\]

\begin{lemma}\label{homglz}
Let $G$ be a group. There exists no group homomorphism $\rho\colon \GL_3(\Z)\to G$ such that $\rho(s_1)$ has order $3$ and $s_2\in\ker(\rho)$.
\end{lemma}

\begin{proof}

Assume that such a $\rho$ exists. Let 
\[
A\coloneqq \left( \begin{array}{ccc}
1 & 0 &0\\
0& 1 &0\\
0& -1& 1\\
 \end{array} \right),
 B\coloneqq \left( \begin{array}{ccc}
-1 & 1 &0\\
0& 0&1\\
1& 0& 0\\
 \end{array} \right), T\coloneqq \left( \begin{array}{ccc}
1 & 0 &0\\
0& -1&0\\
0& 0& -1\\
 \end{array} \right)\in\GL_3(\Z).
\]

One calculates $(A(s_2(Bs_2B^{-1}))A^{-1})=s_1T$. So $s_1T$ is contained in the kernel of $\rho$ and we get $\rho(T)=\rho(s_1^{-1})$. But this is a contradiction since the order of $T$ is 2.
\end{proof}

The following construction comes up in the context of tetrahedral line complexes (see \cite{MR2964027}). Consider the 4 hyperplanes in $\p^3$ 
\[
E_0:=\{x_0=0\},\, E_1:=\{x_1=0\}, \,E_2:=\{x_2=0\}, \,E_3:=\{x_3=0\}.
\]
A line $l\in\mathbb{G}(1,3)$ that is not contained in any of the $E_i$, intersects each plane $E_i$ in one point $p_i$. We thus obtain a rational surjective map
\[
cr\colon\mathbb{G}(1,3)\dashrightarrow\p^1
\]
that is defined by associating to the line $l$ the cross ratio between the points $p_i$.  

The closure $\overline{cr^{-1}([a:b])}$ in $\mathbb{G}(1,3)$ is irreducible if and only if $[a:b]\in \p^1\setminus \{[0:1], [1:0], [1:1]\}$, whereas $\overline{cr^{-1}([a:b])}$ consists of two irreducible components in all the other cases (\cite[Chapter 10.3.6]{MR2964027}).

Recall that  $\alpha$ is the automorphism of $\PGL_4$ given by $g\mapsto {}^tg^{-1}$.

\begin{proposition}\label{nograss}
There exists no non-trivial group homomorphism 
\[
\Phi\colon \left<\PGL_4(\C), \W_3\right>\to\Bir(\mathbb{G}(1,3))
\]
such that $\Phi(\PGL_4(\C))\subset\aut(\mathbb{G}(1,3))$.

In particular, neither the action of $\PGL_4(\C)$ on $\mathbb{G}(1,3)$ given by the embedding $\varphi_G$ (see Example $\ref{grass}$) nor the action given by  $\varphi_G\circ\alpha$ can be extended to a rational action of $\Cr_4$.
\end{proposition}

\begin{proof}
The proof of Corollary \ref{hninkernel} implies that if $\PGL_4(\C)$ is contained in the kernel of a homomorphism $\Phi\colon\left<\PGL_4(\C), \W_3\right>\to\Bir(\mathbb{G}(1,3))$, then $\Phi$ is trivial. So we may assume that $\Phi$ embeds $\PGL_4(\C)$ into $\aut^0(\mathbb{G}(1,3))$. By Theorem \ref{pglmain}, it is therefore enough to show that $\varphi_G$ and $\varphi_G\circ\alpha$ do not extend to a homomorphism of $ \left<\PGL_4(\C), \W_3\right>$.

The $\varphi_G(D_3)$-orbit of a line that is not contained in one of the planes $E_i$ and that does not pass through any of the points $[1:0:0:0], [0:1:0:0], [0:0:1:0], [0:0:0:1]$, has dimension $3$ and these are all $\varphi_G(D_3)$-orbits of dimension $3$. 

Since $\varphi_G(D_3)$ stabilizes the hyperplanes $E_i$ and since the cross ratio is invariant under linear transformations, we obtain that $cr$ is $\varphi_G(D_3)$-invariant. By the above remark, the rational map $cr$ therefore parametrizes all but finitely many $\varphi_G(D_3)$-orbits of dimension 3 by $\p^1\setminus \{[0:1], [1:0], [1:1]\}$.

The image $\varphi_G(\s_4)$, where $\s_4\subset\PGL_4$ is the subgroup of coordinate permutations, normalizes $\varphi_G(D_3)$ and therefore it permutes its 3-dimensional orbits. Since $\s_4$ permutes the hyperplanes $E_i$, we can describe its action on the 3-dimensional $\varphi_G(D_3)$-orbits by its action on the cross ratio of the intersection of general lines with the planes $E_i$.

Let $r$ be the cross ratio between the points $p_0, p_1, p_2, p_3$ on a line. One calculates that the cross ratio between $p_3, p_1, p_2, p_0$ is again $r$ and that the cross ratio between the points $p_2, p_1, p_2, p_3$ is $\frac{1}{1-r}$.  Hence the image of $\tau_1:=[x_3:x_1:x_2:x_0]$ leaves $cr$ invariant, whereas for the permutation $\tau_2:=[x_2:x_1:x_0:x_3]$ we have $cr\circ\varphi(\tau_2)\neq cr$ and $cr\circ\varphi(\tau_2)^2\neq cr$. 

Let $f\colon\mathbb{G}(1,3)\dashrightarrow\p^4$ be a birational transformation and let $\varphi_G'\colon=f\circ\varphi_G\circ f^{-1}$. The image $\varphi_G'(D_3)\subset \Cr_4$ is an algebraic torus of rank 3 and therefore, by Proposition \ref{poptori}, conjugate to the standard subtorus $D_3\subset D_4$ of rank 3. In other words, there exists a rational map $\p^4\dashrightarrow \p^1$
whose fibers consist of the closure of the $\varphi_G'(D_3)$-orbits. The image $\varphi_G'(\s_4)$ permutes the torus orbits, hence we obtain a homomorphism
$\rho\colon\s_4\to\PGL_2(\C).$
By what we observed above, the permutation $\tau_1$ is contained in the kernel of $\rho$, whereas the image $\rho(\tau_2)$ has order 3. The matrix representation in $\GL_3(\Z)$ of $\tau_1$ corresponds to $s_1$ and the matrix representation of $\tau_2$ corresponds to $s_2$.

It follows now from Lemma \ref{homglz} that $\rho$ can not be extended to a homomorphism from $\GL_3(\Z)\simeq\W_3$ to $\PGL_2(\C)$, which implies that there exists no homomorphism $\Phi\colon\left<\PGL_4(\C), \W_3\right>\to \Cr_4$ such that $\Phi(\PGL_4(\C))=\varphi'_G(\PGL_4(\C))$, since $\W_3$ normalizes the torus and its image would therefore permute the torus orbits as well. The statement follows.
\end{proof}

\subsection{The case $\p(T\p^2)$}

Recall that matrices of order two in $\PGL_2(\C)$  have the form
\begin{equation}\label{order2}
 \left[ \begin{array}{cc}
0 &  1\\
a & 0 \\
 \end{array} \right]
,\text{ or }
\left[ \begin{array}{cc}
1 &  b\\
c & -1 \\
 \end{array} \right],\, \text{ where } a\in\C^*,\, b,c\in\C,\, bc\neq-1.
\end{equation}

\begin{proposition}\label{varphib}
The embedding $\varphi_B\colon \PGL_3(\C)\to\Bir(\p(T\p^2))$ extends uniquely to an embedding\[
\Phi_B\colon\Cr_2\to\Bir(\p(T\p^2)).
\]
\end{proposition}

\begin{proof}
It is enough to show that every extension coincides withe one given in Example \ref{defb}. For this it is enough to show that the image of $\sigma$ is uniquely determined. Assume that there is an extension $\psi\colon\Cr_2\to\Bir(\p(T\p^2))$ of $\varphi_B$. We will show $\psi=|\Psi_B$.

Let $d\in D_2$, $d=(ax_1,bx_2)$ with respect to affine coordinates  given by $x_0=1$. Then
$ \varphi_B(d)=(ax_1,bx_2, (b/a)x_3)$,
with respect to suitable local affine coordinates of $\p(T\p^2)$.
Let 
$\phi\colon\p(T\p^2)\dasharrow\p^2\times\p^1$ 
be the birational map given by
\[
\phi\colon (x_1, x_2, x_3)\mapsto (x_1,x_2,\frac{x_1}{x_2}x_3),
\]
with respect to local affine coordinates.

Let $\psi_1\colon\Cr_2\to\Bir(\p^2\times\p^1)$ be the algebraic embedding $\psi_1=\phi\circ\psi\circ\phi^{-1}$. This gives us a $\p^2$-fibration, which we call the horizontal fibration, and a $\p^1$-fibration, which we call the vertical fibration. The image $\psi_1(D_2)$ acts canonically on the first factor and leaves the second one invariant. The horizontal fibers thus consist of the closures of $D_2$-orbits. Since $\W_2$ normalizes $D_2$, the image $\psi_1(\W_2)$ permutes the orbits of $\psi_1(D_2)$. Hence it preserves the horizontal fibration and we obtain a homomorphism
\[
\rho\colon W_2\simeq\GL_2(\Z)\to\Bir(\p^1)=\PGL_2(\C). 
\]
In what follows we identify $\W_2$ with $\GL_2(\Z)$.

The images of the three transpositions in $\s_3=\W_2\cap\PGL_3(\C)$ under $\rho$ are:
\[
\rho\left(\begin{array}{cc}
0 & 1 \\
1 & 0 \\
 \end{array} \right)=\left[\begin{array}{cc}
0 & 1 \\
1 & 0 \\
 \end{array} \right], \rho\left( \begin{array}{cc}
1 & -1 \\
0 & -1 \\
 \end{array}\right)=\left[ \begin{array}{cc}
1 & -1 \\
0 & -1 \\
 \end{array}\right]
\]
\[
\text{ and }  \rho\left( \begin{array}{cc}
-1 & 0 \\
-1& 1 \\
 \end{array}\right)=\left[ \begin{array}{cc}
-1 & 0 \\
-1 & 1 \\
 \end{array}\right].
\]

The image $\rho(\sigma)$ is either the identity or it has order 2. The elements of the form~(\ref{order2}) do not commute with the images of $\s_3$ described above. Since $\sigma$ is contained in the center of $\W_2$, we obtain $\rho(\sigma)=\id$. 

It remains to show that the action of $\psi_1(\sigma)$ on the first factor of $\p^2\times\p^1$ is the standard action.
Let $M= \p^2$ be a horizontal fiber. It is stabilized by $\psi(D_2)$ and $\psi(\sigma)$, so we obtain a homomorphism 
\[
\gamma\colon\left<D_2,\sigma\right>\to \Bir(M)=\Cr_2.
\] 
Since $\sigma d\sigma^{-1}=d^{-1}$ for all $d\in D_2$, there exists a $d\in D_2$ such that $\gamma(\sigma)=d\sigma$. This is true for all horizontal fibers, so $\psi_1(\sigma)$ induces an automorphism of $U\times \p^1$, where 
 \[
 U=\{[x_0:x_1:x_2]\mid x_0, x_1, x_2\neq 0\}\subset\p^2.
 \] 

Let $S\simeq\p^1\subset U\times\p^1$ be a vertical fiber and $\pi\colon U\times\p^1\to U$ the projection onto the first factor. Then $\pi\circ\psi_1(\sigma)(S)$ is a regular map from $\p^1$ to the affine set $U$ and is therefore constant. We obtain that $\psi_1(\sigma)$ preserves the vertical fibration.

The image $\psi_1(\PGL_3(\C))$ preserves the vertical fibration as well and projection onto $\p^2$ yields a homomorphism from $\PGL_3(\C)$ to $\Cr_2$ that is the standard embedding. Hence $\psi_1(\Cr_2)$ preserves the vertical fibration and we  obtain an algebraic homomorphism from $\Cr_2$ to $\Cr_2$, which is uniquely determined by its restriction to $\PGL_3(\C)$. So the image $\psi_1(\sigma)$ is uniquely determined.
\end{proof}

\begin{lemma}
There exists no homomorphism $\Phi\colon\Cr_2\to\Cr_3$ such that 
\[
\Phi|_{\PGL_3(\C)}=\varphi_B\circ\alpha,
\]
where $\varphi_B$ denotes the embedding of $\PGL_3$ into $\Cr_3$ from Example $\ref{defb}$ and $\alpha$ the algebraic automorphism of $\PGL_3$ given by $g\mapsto {}^tg^{-1}$.
\end{lemma}

\begin{proof}
Assume that such an extension $\Phi\colon \Cr_2\to\Cr_3$ of $\varphi_B\circ\alpha$ exists. 

Observe that $\alpha(D_2)=D_2$ and that $\alpha|_{\s_3}=\id_{\s_3}$. Therefore, we can repeat the same argument as in the proof of Proposition \ref{varphib} to obtain $\Psi(\sigma)=\Phi_B(\sigma)$. But we have 
\[
\Psi(\sigma) \Psi(g)\Psi(\sigma)\Psi( g)\Psi(\sigma) \Psi(g)\neq\id
\]
for $g=[z-x: z-y:z]$ - this contradicts the relations in $\Cr_2$ (Proposition \ref{relations}).
\end{proof}

\subsection{The case $\p(\mathcal{O}_{\p^n}\oplus\mathcal{O}_{\p^n}(-k(n+1)))$}

\begin{proposition}\label{phil}
The algebraic homomorphism $\varphi_l\colon \PGL_{n+1}(\C)\to\Bir(F_l)$ extends uniquely to the embedding
\[
\Psi_l\colon\left<\PGL_{n+1}(\C), \W_n\right>\to\Bir(F_l) \text{ (see Example $\ref{defl}$). }
\]
\end{proposition}

\begin{proof}
Suppose that there is an extension $\psi\colon\G_n\to\Bir(F_{l})$ of $\varphi_l$. We will show that $\psi$ is unique and therefore that $\psi=\Psi_l$.

Let $(x_1,\dots, x_{n-1}, w)$ be local affine coordinates of $F_{l}$ such that for every $g\in\PGL_{n+1}(\C)$ the image $\varphi_l(g)$ acts by
\[
(x_1,\dots, x_{n}, w)\mapsto (g(x_1,\dots, x_{n}), J(g(x_1,\dots, x_{n}))^{-l}w).
\] 
In particular, the image under $\psi$ of $(d_1x_1,\dots, d_{n}x_{n})\in D_{n}$ acts by
\[
(x_1,\dots, x_{n}, w)\mapsto (d_1x_1,\dots, d_{n}x_{n}, (d_1\cdots d_n)^{-l}w).
\]
Define $\phi\colon F_l\to\p^{n}\times\p^1$ by
\[
\phi\colon (x_1,\dots, x_{n}, w)\mapsto (x_1,\dots, x_{n},(x_1\cdots x_{n})^{l}w)
\]
with respect to local affine coordinates.
Let $\psi_1\colon\Cr_{n}\to\Bir(\p^{n}\times\p^1)$ be the algebraic embedding $\psi_1\coloneqq\phi\circ\psi\circ\phi^{-1}$. Then the image $\psi_1(D_{n})$ acts canonically on the first factor and leaves the second one invariant. Since $\W_n$ normalises $D_{n}$, the image $\psi_1(\W_{n})$ permutes the orbits of $\psi_1(D_{n})$. Hence $\psi_1(\W_n)$ preserves the horizontal fibration. This induces a homomorphism
\[
\rho\colon\W_n\simeq\GL_{n}(\Z)\to\PGL_2(\C).
\]
In what follows, we identify $\W_n$ with $\GL_n(\Z)$.
Let $\mathcal{A}_{n+1}\subset\s_{n+1}\subset\PGL_{n+1}(\C)$
 be the subgroup of coordinate permutations $s\in \s_{n+1}$ such that $J(s)=1$. Hence $\mathcal{A}_{n+1}\in\ker(\rho)$. 
Note that the fixed point set of $\psi_1(\mathcal{A}_{n+1})$ is the vertical fiber  
\[
L:=[1:\dots:1]\times \p^1\subset\p^{n}\times\p^1.
\] 

Since $\sigma_{n}$ commutes with $\mathcal{A}_{n+1}$, the image $\psi_1(\sigma_{n})$ stabilises  $L$. The group $\psi_1(D_{n})$ acts transitively on an open dense subset of vertical fibers that contains $L$. Since $\psi_1(\sigma_{n})$ normalizes $\psi_1(D_n)$, we obtain that $\psi_1(\sigma_{n})$ preserves the vertical fibration. Therefore  $\left<\PGL_{n+1}(\C),\sigma_{n}\right>$ preserves the vertical fibration. We obtain a homomorphism
$\left<\PGL_{n+1}(\C),\sigma_{n}\right>\to\Cr_n,$
which is, by Corollary \ref{autgn} and its proof, the standard embedding. 

Let 
\[
f_A=(\frac{1}{x_1},x_2,\dots, x_n).
\]
In \cite{Blanc:2014aa} it is shown that $f_A$ is contained in $\left<\PGL_{n+1}(\C),\sigma_{n}\right>$, which implies that $\psi_1(f_A)$ preserves the vertical fibration and that its action on $\p^{n}$ is the standard action.

Recall that $(hf_A)^3=\id$ for $h=(1-x_1,x_2,\dots, x_{n-1})\in\Cr_{n}$. The image $\psi_1(h)$ is
\[
\psi_1(h)\colon(x_1,\dots, x_{n}, z)\mapsto (1-x_1, x_2,\dots, x_{n}, (-1)^lz).
\]
Denote by $A\in\GL_n(\Z)$ the integer matrix corresponding to $f_A$. We have $\rho(A)=\id$ or $\rho(A)$ is of order two, i.e. has the form (\ref{order2}). 

Suppose that $\rho(A)=\id$. Then 
\[
\psi_1(f_A)\colon (x_1,\dots x_{n}, z)\mapsto (\frac{1}{x_1}, x_2\dots x_{n}, z).
\]
The relation $(hf_A)^3=\id$ then implies that $l$ is even.

Suppose that 
\[
\rho(f_A)=\left[ \begin{array}{cc}
1 &  b\\
c & -1 \\
 \end{array} \right], \text{ where } b,c\in\C, bc\neq-1,\] hence
\[
\psi_1(f_A)\colon (x_1,\dots x_{n}, z)\mapsto \left(\frac{1}{x_1}, x_2\dots x_{n},\frac{z+b}{cz-1}\right)
\]
and therefore
\[
\psi_1(hf_A)= (x_1,\dots x_{n}, z)\mapsto \left(1-\frac{1}{x_1},\dots x_{n},\frac{(-1)^lz+(-1)^lb}{cz-1}\right),
\]
One calculates that if $l$ is even, then the relation $(hf_A)^3=\id$ is not satisfied. So assume that $l$ is odd. This gives
\[
\psi_1(hf_A)^3= (x_1,\dots x_{n}, z)\mapsto \left(x_1,\dots x_{n},\frac{a_1z+a_2}{a_3z-a_4}\right),
\]
where $a_1=3bc-1, a_2=(bc-1)b-2b, a_3=(1-bc)c+2c$ and $a_4=3bc-1$. So $(hf_A)^3=\id$ yields either $l$ odd and $b=c=0$ or $l$ odd and $bc=3$. However, the latter is not possible. Consider the transformation 
\[
\tau=(x_1,\dots, x_{n-2}, x_n, x_{n-1})\in\s_n.
\]
We have $f_A\tau=\tau f_A$. Note that
\[
\psi_1(\tau)\colon(x_1,\dots, x_{n}, z)\mapsto (x_1,\dots, x_{n-2}, x_n, x_{n-1},\dots, x_{n}, (-1)^lz)
\]
and this transformation does not commute with $\left(x_1,\dots x_{n},\frac{a_1z+a_2}{a_3z-a_4}\right)$ in the second case. Hence $c=b=0$ and $l$ is odd.

Finally, assume that
 \[
 \rho(f_A)=\left[ \begin{array}{cc}
0 &  1\\
a & 0 \\
 \end{array} \right], \text{ where } a\in\C^*.
 \]
This implies 
\[
\psi_1(f_A)\colon  (x_1,\dots x_{n}, z)\mapsto \left(\frac{1}{x_1}, x_2\dots x_{n},\frac{1}{az}\right)
\]
and hence $\psi(hf_A)^3\neq \id$. 

We conclude that
\[
\rho(f_A)=\left[ \begin{array}{cc}
1 &  0\\
0 & (-1)^l \\
 \end{array} \right]
 \]
and therefore that the action of $\psi(f_A)$ is uniquely determined by $l$. Hence 
\[
\psi|_{\left<\PGL_n(\C), \sigma_{n-1}\right>}=\Psi_l|_{\left<\PGL_n(\C), \sigma_{n-1}\right>}.
\]

Let $f_B, f_C, f_D$ and $f_E\in\Cr_n$ be as in the proof of Corollary \ref{hninkernel}. By Lemma \ref{gengn} it remains to show that the image $\psi(f_B)$ is uniquely determined. We use once more the relation 
\[
f_B=f_Df_Cf_Ef_D^{-1}.
\] 
Since $\rho(CE)=\id$ and since $f_D$ has order two, we obtain $\rho(B)=\id$. 

Let $c\in\p^1$ such that the restriction of $\psi_1(f_B)$ to the hyperplane 
\[
\{c\}\times\p^{n}\subset\p^1\times\p^{n}
\]
 is a birational map. Then the restriction of $\psi_1(f_B)$ to  $\{c\}\times\p^{n}$ has to fulfill the relations with the group $\left<\PGL_{n+1}(\C),\sigma_{n}\right>$. By Corollary \ref{autgn} we obtain that this restriction has to be $f_B$.  Hence the image $\psi_1(f_B)$ is unique.
\end{proof}

\begin{proposition}
There exists no group homomorphism $\psi\colon \G_{n}\to\Bir(F_l)$ such that $\psi|_{\PGL_{n+1}(\C)}=\varphi_l\circ\alpha$.
\end{proposition}

\begin{proof}
Assume that such an extension $\psi\colon\G_{n}\to\Cr_n$ exists. 
Let $\phi$ be as in Proposition \ref{phil} and 
$\psi_2\colon\G_{n}\to\Bir(\p^1\times\p^{n}),$
\[
\psi_2\coloneqq\phi\circ\varphi_l\circ\alpha\circ\phi^{-1}.
\] 
Similarly as in the proof of Proposition \ref{phil} one can show that $\psi_2(\sigma_{n})$ preserves the vertical fibration. In that way we obtain an algebraic homomorphism 
\[
A\colon \left<\PGL_{n+1},\sigma\right>\to\Cr_{n}
\]
such that $A|_{\PGL_{n+1}(\C)}=\alpha$. Such a homomorphism does not exist by Corollary \ref{autgn}.
\end{proof}

\subsection{The case $C\times\p^n$}

\begin{proposition}\label{phic}
The embedding $\varphi_C\colon \PGL_{n+1}(\C)\to\Bir(C\times\p^{n})$ extends uniquely to the standard embedding   
\[
\Phi_C\colon\G_{n}\to\Bir(C\times\p^{n}) \text{ (see Example $\ref{phic}$). }
\]
\end{proposition}

\begin{proof}
Suppose that there is an extension 
$\Psi\colon\G_{n}\to\Bir(C\times\p^{n})$ of $\varphi_C$. 
By definition, $\Psi(\PGL_{n+1}(\C))$ fixes the horizontal fibration with fibers isomorphic to $\p^n$. Moreover, each of the horizontal fibers is a closure of a $\Psi(D_{n})$-orbit. Since the elements of $\W_{n}$ commute with $D_{n}$, we conclude that $\Psi(\W_n)$ preserves the horizontal fibration. Hence $\G_{n}$ preserves the horizontal fibration and we obtain a homomorphism
 \[
 \rho\colon\G_{n}\to\Bir(C)
 \]
  such that $\PGL_{n+1}(\C)\subset\ker(\rho)$. In the Appendix it is shown that the normal subgroup generated by $\PGL_{n+1}$ in $\G_n$ is all of $\G_n$. Hence $\rho$ is trivial and $\Psi(\G_{n})$ fixes the horizontal fibration. The restriction $\Psi(\G_{n})|_{c\times\p^{n}}$ for any $c\in C$ defines a homomorphism from $\G_{n}$ to $\Cr_{n}$ such that the restriction to $\PGL_{n+1}(\C)$ is the standard embedding. By Corollary \ref{autgn}, this is the standard embedding. Hence $\Psi$ is unique.
\end{proof}

\begin{proposition}
There exists no group homomorphism $\Psi\colon \G_{n}\to\Bir(\C\times\p^{n})$ such that $\Psi|_{\PGL_{n+1}(\C)}=\varphi_C\circ\alpha$.
\end{proposition}

\begin{proof}
Assume there exists such a $\Psi$. As in the proof of Proposition \ref{phic} one can show that $\Psi(\G_{n})$ fixes the horizontal fibration. The restriction $\Psi(\G_{n})|_{c\times\p^{n-1}}$ defines for each $c\in C$ a homomorphism from $\G_{n}$ to $\Cr_{n}$ such that the restriction to $\PGL_{n+1}(\C)$ is given by $g\mapsto \alpha(g)$.  By Corollary \ref{autgn}, there exists no such homomorphism. 
\end{proof}


\appendix
\section*{Appendix}

\subsection*{Relations and structures in $\Cr_n$}\label{sectionrelation}\label{appendix}
\renewcommand{\thesection}{A} 
We will often use the following relations between elements of the Cremona group:
\begin{lemma}\label{relations}
In $\Cr_2$ the following relations hold:
\begin{enumerate}
 \item$\sigma\tau(\tau\sigma)^{-1}=\id$ for all $\tau\in\s_3$, 
 \item$\sigma d=d^{-1}\sigma$ for all diagonal maps $d\in D_2$ and 
\item\label{rel}$(\sigma h)^3=\id$ for $h=[x_2-x_0:x_2-x_1:x_2]$.
\end{enumerate}
\end{lemma}

\begin{proof}
One checks the identities by direct calculation.
\end{proof}

Denote by $\Cr_n^0\subset\Cr_n$ the subgroup consisting of elements that contract only rational hypersurfaces. We have $\G_n\subset\Cr_n^0$. On the other hand, it seems to be an interesting open question, whether there exist elements in $\Cr_n^0$ that are not contained in $\G_n$ for any $n\geq 3$ (cf. \cite{MR3229349}).

\begin{lemma}\label{gengn}
The group $\G_n$ is generated by $\PGL_{n+1}(\C)$ and the birational transformations
$\sigma_n:=(x_1^{-1}, x_2^{-1},\dots, x_n^{-1}) \text{ and } f_B:=(x_1x_2, x_3,\dots, x_n).$
\end{lemma}

\begin{proof}
It is known that $\GL_n(\Z)$ is generated by the subgroup of permutation matrices in $\GL_n(\Z)$ and the two elements
\[
A:=\left( \begin{array}{ccccc}
-1 & 0&0&\cdots&0 \\
0 & 1&0&\cdots&0 \\
&&\cdots&&\\
\\
0&0&\cdots&1&0\\
0&0&\cdots&0&1
 \end{array}\right) \text{ and }
B:= \left( \begin{array}{ccccc}
1 & 1&0&\cdots&0 \\
0 & 1&0&\cdots&0 \\
&&\cdots&&\\
\\
0&0&\cdots&1&0\\
0&0&\cdots&0&1
 \end{array}\right)
\]
(see for example \cite[III.A.2]{MR1786869}). Notice that $f_B$ is the birational transformation in $\W_n$ corresponding to $B$. 
Let $f_A$ be the birational transformation corresponding to $A$. In \cite{Blanc:2014aa} it is shown that $f_A$ is contained in $\left<\PGL_{n+1}(\C),\sigma_n\right>$.
\end{proof}

The goal of this appendix is to prove the following two corollaries of Theorem~\ref{thmserge}:
 
\begin{corollary}\label{hninkernel}
Let $n>m$ and let $\Phi\colon\Cr_n\to\Cr_m$ be a group homomorphism. Then the normal subgroup of $\Cr_n$ containing $\G_n$ is contained in the kernel of $\Phi$.
\end{corollary}
No such non-trivial homomorphism is known so far. In fact, it is an open question, whether $\Cr_n$ is simple for $n\geq 3$.

\begin{corollary}\label{autgn}
Let $\Psi\colon\G_n\to\Cr_ng$ be a non-trivial group homomorphism. Then there exists an automorphism of fields $\gamma$ of $\C$ and an element $g\in\Cr_n$ such that $\gamma(g\Psi g^{-1})$ is the standard embedding.

Moreover, the extension of the standard embedding $\varphi\colon\PGL_{n+1}\to\Cr_n$ to the group $\G_n$ is unique. The embedding $\varphi\circ\alpha$, where $\alpha\colon\PGL_{n+1}\to\PGL_{n+1}$ is the algebraic automorphism $g\mapsto {}^tg^{-1}$, does not extend to a homomorphism from $\G_n$ to $\Cr_n$.
\end{corollary}

By the theorem of Noether and Castelnuovo, Corollary \ref{autgn} implies in particular the theorem of D\'eserti about automorphisms of $\Cr_2$.

\begin{proof}[Proof of Corollary $\ref{hninkernel}$]
By Lemma \ref{gengn} it is enough to show that $\sigma_n$ and $f_B$ are contained in the normal subgroup containing $\PGL_{n+1}(\C)$. Let
\[
g_n:=[x_n-x_0:x_n-x_1:\dots:x_n-x_{n-1}:x_n]\in\PGL_{n+1}(\C).
\]
Then 
$\sigma_ng_n\sigma_ng_n\sigma_ng_n=\id.$
In particular, $\sigma_ng_n$ conjugates $\sigma_n$ to $g_n$. 

Let 
\[
C:=\left( \begin{array}{ccccc}
-1 & 2&0&\cdots&0 \\
0 & 1&0&\cdots&0 \\
&&\cdots&&\\
\\
0&0&\cdots&1&0\\
0&0&\cdots&0&1
 \end{array}\right), 
D:=\left( \begin{array}{ccccc}
-1 & 0&0&\cdots&0 \\
-1& 1&0&\cdots&0 \\
&&\cdots&&\\
\\
0&0&\cdots&1&0\\
0&0&\cdots&0&1
 \end{array}\right), 
 \]
 \[E:=\left( \begin{array}{ccccc}
0& 1&0&\cdots&0 \\
1& &0&\cdots&0 \\
&&\cdots&&\\
\\
0&0&\cdots&1&0\\
0&0&\cdots&0&1
 \end{array}\right)\in\GL_n(\Z)
\]
and let $f_C, f_D$ and $f_E$ be the corresponding elements in $\W_n$. It is shown in \cite{Blanc:2014aa} that $f_C$ is contained in $\left<\PGL_{n+1}(\C),\sigma_n\right>$. Moreover, one calculates that
\[
f_B=f_Df_Cf_Ef_D^{-1},
\]
which implies that $f_B$ is conjugate to an element in $\left<\PGL_{n+1}(\C),\sigma_n\right>$.
\end{proof}

\begin{proof}[Proof of Corollary $\ref{autgn}$]
By Theorem \ref{thmserge} we may assume that, up to conjugation and automorphism of fields, the restriction of $\Psi$ to $\PGL_{n+1}(\C)$ ist the standard embedding or the standard embedding composed with the automorphism $\alpha$ of $\PGL_{n+1}(\C)$ given by $\alpha(g)={}^tg^{-1}$. 

In particular, $\Psi(D_n)=D_n$ and therefore $\Psi(\W_n)$ is contained in $D_n\rtimes\W_n$. Assume that $\Psi(\sigma_n)=d\tau$ for some $d\in D_n$ and $\tau\in\W_n$. The relation 
$\Psi(\sigma_n)e\Psi(\sigma_n)=e^{-1}$
 for all $e\in D_n$ implies $\tau=\sigma_n$. Note that the restriction of $\Psi$ to $\s_{n+1}$ is the standard embedding. So for all  $\tau\in\s_{n+1}$ we obtain
\[
\tau d\sigma_n=d\sigma_n \tau=d\tau\sigma_n.
\]
The only element in $D_n$ that commutes with $\s_{n+1}$ is the identity. Hence $\Psi(\sigma_n)=\sigma_n$.

Let $g_n$ be as in the proof of Corollary  \ref{hninkernel}. The relation 
$\sigma_ng_n\sigma_ng_n\sigma_ng_n=\id$
 implies that $\Psi|_{\PGL_{n+1}(\C)}$ is the standard embedding, since 
\[
\sigma_n\alpha(g_n)\sigma_n\alpha(g_n)\sigma_n\alpha(g_n)\neq\id.
\]

It remains to show that $\Psi(f_B)=f_B$. Let $d\in D_n$, and $\rho\in \W_n$ such that $\Psi(f_B)=d\rho$. The image $\Psi(f_B)$ acts on $\Psi(D_n)$ by conjugation. We have that $\Psi(D_n)=D_n$, so the action of $\Psi(f_B)$ on $\Psi(D_n)$ is determined by $\rho$. Since  $\Psi|_{D_n}$ is the standard embedding, we obtain $\rho=f_B$. Let 
$d=(d_1x_1,\dots, d_n x_n).$
The image $\Psi(f_B)$ commutes with $\sigma_n$. We obtain
\[
d^{-1}\sigma_nf_B=\sigma_ndf_B=df_B\sigma_n=d\sigma_n f_B
\]
and hence $d_i=\pm 1$ for all $i$. 

The image $\Psi(f_B)$ commutes with all elements of $\s_{n+1}$ that fix the coordinates $x_1$ and $x_2$. Similar as above, this yields that $d$ commutes with all elements of $\s_{n+1}$ that fix the coordinates $x_1$ and $x_2$ and we get $d_i=1$ for $i\neq 1$ and $i\neq 2$.

In \cite{Blanc:2014aa} it is shown that $f_B^2$ is contained in $\left<\PGL_{n+1}(\C),\sigma_n\right>$. By what we proved above, this gives 
\[
\Psi(f_B^2)=f_B^2=df_Bdf_B=dd'f_B^2,
\]
 where
$d'=(d_1d_2x_1,d_2x_2,\dots,d_n x_n).$
So $dd'=\id$, which yields $d_1^2d_2=1$ and therefore $d_2=1$. This means that we have either $\Psi(f_B)=f_B$ or $\Psi(f_B)=df_B$ with $d=(-x_1,x_2,\dots, x_n)$. 

Let 
\[
r_1:=[x_0:x_1:\dots:x_{n-1}:x_n+x_1], r_2:=[x_n:x_1:\dots:x_{n-1}:x_0], \]\[r_3:=[x_n:x_0:x_2:\dots:x_{n-1}:x_1], t:=[x_n:x_0:\dots:x_{n-1}].
\]
We have the relation
\[
(r_2tf_Bt^{-1}r_3)r_1(r_2tf_Bt^{-1}r_3)=r_1
\]
and therefore
\[
(r_2t\Psi(f_B)t^{-1}r_3)r_1(r_2t\Psi(f_B)t^{-1}r_3)=r_1.
\]
One calculates that if $\Psi(f_B)=(-x_1,x_2,\dots,x_n)f_B$ then this relation is not satisfied. Hence $\Psi(f_B)=f_B$.
\end{proof}


\bibliographystyle{amsalpha}
\bibliography{/Users/christian/Dropbox/Literatur/bibliography_cu}

\providecommand{\bysame}{\leavevmode\hbox to3em{\hrulefill}\thinspace}
\providecommand{\MR}{\relax\ifhmode\unskip\space\fi MR }
\providecommand{\MRhref}[2]{%
  \href{http://www.ams.org/mathscinet-getitem?mr=#1}{#2}
}
\providecommand{\href}[2]{#2}
\begin{thebibliography}{BCTSSD85}

\bibitem[AC02]{MR1874328}
Maria Alberich-Carrami{\~n}ana, \emph{Geometry of the plane {C}remona maps},
  Lecture Notes in Mathematics, vol. 1769, Springer-Verlag, Berlin, 2002.
  \MR{1874328 (2002m:14008)}

\bibitem[Akh95]{MR1334091}
Dmitri~N. Akhiezer, \emph{Lie group actions in complex analysis}, Aspects of
  Mathematics, E27, Friedr. Vieweg \& Sohn, Braunschweig, 1995. \MR{1334091
  (96g:32051)}

\bibitem[BB66]{MR0200279}
A.~Bia{\l}ynicki-Birula, \emph{Remarks on the action of an algebraic torus on
  {$k^{n}$}}, Bull. Acad. Polon. Sci. S{\'e}r. Sci. Math. Astronom. Phys.
  \textbf{14} (1966), 177--181. \MR{0200279 (34 \#178)}

\bibitem[BCTSSD85]{MR786350}
Arnaud Beauville, Jean-Louis Colliot-Th{{\'e}}l{{\`e}}ne, Jean-Jacques Sansuc,
  and Peter Swinnerton-Dyer, \emph{Vari{\'e}t{\'e}s stablement rationnelles non
  rationnelles}, Ann. of Math. (2) \textbf{121} (1985), no.~2, 283--318.
  \MR{786350 (86m:14009)}

\bibitem[BD15]{MR3410471}
J{{\'e}}r{{\'e}}my Blanc and Julie D{{\'e}}serti, \emph{Degree growth of
  birational maps of the plane}, Ann. Sc. Norm. Super. Pisa Cl. Sci. (5)
  \textbf{14} (2015), no.~2, 507--533. \MR{3410471}

\bibitem[BF13]{MR3092478}
J{{\'e}}r{{\'e}}my Blanc and Jean-Philippe Furter, \emph{Topologies and
  structures of the {C}remona groups}, Ann. of Math. (2) \textbf{178} (2013),
  no.~3, 1173--1198. \MR{3092478}

\bibitem[BH14]{Blanc:2014aa}
J{\'e}r{\'e}my Blanc and Isac Hed{\'e}n, \emph{The group of cremona
  transformations generated by linear maps and the standard involution}.

\bibitem[Bla06]{MR2215969}
J{{\'e}}r{{\'e}}my Blanc, \emph{Conjugacy classes of affine automorphisms of
  {$\Bbb K^n$} and linear automorphisms of {$\Bbb P^n$} in the {C}remona
  groups}, Manuscripta Math. \textbf{119} (2006), no.~2, 225--241. \MR{2215969
  (2006m:14015)}

\bibitem[Bla09]{MR2504924}
\bysame, \emph{Sous-groupes alg{\'e}briques du groupe de {C}remona}, Transform.
  Groups \textbf{14} (2009), no.~2, 249--285. \MR{2504924 (2010b:14021)}

\bibitem[Bla13]{MR3080816}
\bysame, \emph{Symplectic birational transformations of the plane}, Osaka J.
  Math. \textbf{50} (2013), no.~2, 573--590. \MR{3080816}

\bibitem[Bri89]{brion1989spherical}
Michel Brion, \emph{Spherical varieties an introduction}, Topological methods
  in algebraic transformation groups, Springer, 1989, pp.~11--26.

\bibitem[Bru97]{MR1474805}
Marco Brunella, \emph{Feuilletages holomorphes sur les surfaces complexes
  compactes}, Ann. Sci. {\'E}cole Norm. Sup. (4) \textbf{30} (1997), no.~5,
  569--594. \MR{1474805 (98i:32051)}

\bibitem[BT73]{MR0316587}
Armand Borel and Jacques Tits, \emph{Homomorphismes ``abstraits'' de groupes
  alg{\'e}briques simples}, Ann. of Math. (2) \textbf{97} (1973), 499--571.
  \MR{0316587 (47 \#5134)}

\bibitem[Can03]{MR2026896}
Serge Cantat, \emph{Endomorphismes des vari{\'e}t{\'e}s homog{\`e}nes},
  Enseign. Math. (2) \textbf{49} (2003), no.~3-4, 237--262. \MR{2026896}

\bibitem[Can14]{MR3230847}
\bysame, \emph{Morphisms between {C}remona groups, and characterization of
  rational varieties}, Compos. Math. \textbf{150} (2014), no.~7, 1107--1124.
  \MR{3230847}

\bibitem[CD16]{Cerveau:2016aa}
Dominique Cerveau and Julie D{\'e}serti, \emph{Birational maps preserving the
  contact structure on $\mathbb{P}^3_\mathbb{C}$}.

\bibitem[CK15]{Corti:2015aa}
Alessio Corti and Anne-Sophie Kaloghiros, \emph{The sarkisov program for mori
  fibred calabi-yau pairs}.

\bibitem[CZ12]{MR3014483}
Serge Cantat and Abdelghani Zeghib, \emph{Holomorphic actions, {K}ummer
  examples, and {Z}immer program}, Ann. Sci. {\'E}c. Norm. Sup{\'e}r. (4)
  \textbf{45} (2012), no.~3, 447--489. \MR{3014483}

\bibitem[Dem70]{MR0284446}
Michel Demazure, \emph{Sous-groupes alg{\'e}briques de rang maximum du groupe
  de {C}remona}, Ann. Sci. {\'E}cole Norm. Sup. (4) \textbf{3} (1970),
  507--588. \MR{0284446 (44 \#1672)}

\bibitem[D{\'e}s06a]{deserti:tel-00125492}
Julie D{\'e}serti, \emph{{On the Cremona group: some algebraic and dynamical
  properties}}, Theses, {Universit{\'e} Rennes 1}, November 2006.

\bibitem[D{\'e}s06b]{MR2278755}
\bysame, \emph{Sur les automorphismes du groupe de {C}remona}, Compos. Math.
  \textbf{142} (2006), no.~6, 1459--1478. \MR{2278755 (2007g:14008)}

\bibitem[D{\'e}s14]{Deserti:2014kq}
Julie D{\'e}serti, \emph{Some properties of the group of birational maps
  generated by the automorphisms of $\mathbb{P}^n_\mathbb{C}$ and the standard
  involution}.

\bibitem[Die71]{MR0310083}
Jean~A. Dieudonn{{\'e}}, \emph{La g{\'e}om{\'e}trie des groupes classiques},
  Springer-Verlag, Berlin-New York, 1971, Troisi{{\`e}}me {{\'e}}dition,
  Ergebnisse der Mathematik und ihrer Grenzgebiete, Band 5. \MR{0310083 (46
  \#9186)}

\bibitem[DL16]{MR3451389}
Jeffrey Diller and Jan-Li Lin, \emph{Rational surface maps with invariant
  meromorphic two-forms}, Math. Ann. \textbf{364} (2016), no.~1-2, 313--352.
  \MR{3451389}

\bibitem[dlH00]{MR1786869}
Pierre de~la Harpe, \emph{Topics in geometric group theory}, Chicago Lectures
  in Mathematics, University of Chicago Press, Chicago, IL, 2000. \MR{1786869
  (2001i:20081)}

\bibitem[Dol12]{MR2964027}
Igor~V. Dolgachev, \emph{Classical algebraic geometry}, Cambridge University
  Press, Cambridge, 2012, A modern view. \MR{2964027}

\bibitem[Enr93]{enriques1893sui}
Federigo Enriques, \emph{Sui gruppi continui di trasformazioni cremoniane nel
  piano}, Rend. Accad. Lincei, 1er sem (1893).

\bibitem[FH91]{MR1153249}
William Fulton and Joe Harris, \emph{Representation theory}, Graduate Texts in
  Mathematics, vol. 129, Springer-Verlag, New York, 1991, A first course,
  Readings in Mathematics. \MR{1153249 (93a:20069)}

\bibitem[Giz99]{MR1714823}
Marat Gizatullin, \emph{On some tensor representations of the {C}remona group
  of the projective plane}, New trends in algebraic geometry ({W}arwick, 1996),
  London Math. Soc. Lecture Note Ser., vol. 264, Cambridge Univ. Press,
  Cambridge, 1999, pp.~111--150. \MR{1714823 (2000i:14018)}

\bibitem[Giz08]{MR2433658}
\bysame, \emph{Klein's conjecture for contact automorphisms of the
  three-dimensional affine space}, Michigan Math. J. \textbf{56} (2008), no.~1,
  89--98. \MR{2433658}

\bibitem[Hud27]{hudson1927}
Hilda~P. Hudson, \emph{Cremona transformation in plane and space}, Cambridge
  University Press, Cambridge, 1927.

\bibitem[Hum75]{MR0396773}
James~E. Humphreys, \emph{Linear algebraic groups}, Springer-Verlag, New
  York-Heidelberg, 1975, Graduate Texts in Mathematics, No. 21. \MR{0396773 (53
  \#633)}

\bibitem[Lam02]{MR1955604}
St{{\'e}}phane Lamy, \emph{Une preuve g{\'e}om{\'e}trique du th{\'e}or{\`e}me
  de {J}ung}, Enseign. Math. (2) \textbf{48} (2002), no.~3-4, 291--315.
  \MR{1955604 (2003m:14099)}

\bibitem[Lam14]{MR3229349}
\bysame, \emph{On the genus of birational maps between threefolds},
  Automorphisms in birational and affine geometry, Springer Proc. Math. Stat.,
  vol.~79, Springer, Cham, 2014, pp.~141--147. \MR{3229349}

\bibitem[Pan99]{MR1686984}
Ivan Pan, \emph{Une remarque sur la g{\'e}n{\'e}ration du groupe de {C}remona},
  Bol. Soc. Brasil. Mat. (N.S.) \textbf{30} (1999), no.~1, 95--98. \MR{1686984
  (2000b:14015)}

\bibitem[Pop13]{MR3135700}
Vladimir~L. Popov, \emph{Tori in the {C}remona groups}, Izv. Ross. Akad. Nauk
  Ser. Mat. \textbf{77} (2013), no.~4, 103--134. \MR{3135700}

\bibitem[Pro07]{MR2265844}
Claudio Procesi, \emph{Lie groups}, Universitext, Springer, New York, 2007, An
  approach through invariants and representations. \MR{2265844 (2007j:22016)}

\bibitem[Ser79]{MR554237}
Jean-Pierre Serre, \emph{Local fields}, Graduate Texts in Mathematics, vol.~67,
  Springer-Verlag, New York-Berlin, 1979, Translated from the French by Marvin
  Jay Greenberg. \MR{554237 (82e:12016)}

\bibitem[Ser08]{serre2008groupe}
Jean-Pierre Serre, \emph{Le groupe de cremona et ses sous-groupes finis},
  S{\'e}minaire Bourbaki \textbf{1000} (2008), 2008--2009.

\bibitem[Sta13]{MR3228629}
Immanuel Stampfli, \emph{A note on automorphisms of the affine {C}remona
  group}, Math. Res. Lett. \textbf{20} (2013), no.~6, 1177--1181. \MR{3228629}

\bibitem[Sum74]{MR0337963}
Hideyasu Sumihiro, \emph{Equivariant completion}, J. Math. Kyoto Univ.
  \textbf{14} (1974), 1--28. \MR{0337963}

\bibitem[Ume82a]{MR676586}
Hiroshi Umemura, \emph{Maximal algebraic subgroups of the {C}remona group of
  three variables. {I}mprimitive algebraic subgroups of exceptional type},
  Nagoya Math. J. \textbf{87} (1982), 59--78. \MR{676586 (84b:14005)}

\bibitem[Ume82b]{MR683251}
\bysame, \emph{On the maximal connected algebraic subgroups of the {C}remona
  group. {I}}, Nagoya Math. J. \textbf{88} (1982), 213--246. \MR{683251
  (84g:14013)}

\bibitem[Ume85]{MR803342}
\bysame, \emph{On the maximal connected algebraic subgroups of the {C}remona
  group. {II}}, Algebraic groups and related topics ({K}yoto/{N}agoya, 1983),
  Adv. Stud. Pure Math., vol.~6, North-Holland, Amsterdam, 1985, pp.~349--436.
  \MR{803342 (87d:14008)}

\bibitem[Ureon]{longversion}
Christian Urech, \emph{On homomorphisms between cremona groups}, Ph.D. thesis,
  University of Basel/ University of Rennes 1, in Preparation.

\bibitem[Wei55]{MR0074083}
Andr{{\'e}} Weil, \emph{On algebraic groups of transformations}, Amer. J. Math.
  \textbf{77} (1955), 355--391. \MR{0074083 (17,533e)}

\bibitem[Wil38]{MR1563724}
A.~R. Williams, \emph{Birational transformations in 4-space and 5-space}, Bull.
  Amer. Math. Soc. \textbf{44} (1938), no.~4, 272--278. \MR{1563724}

\bibitem[Zai95]{MR1389430}
Dmitri Zaitsev, \emph{Regularization of birational group operations in the
  sense of {W}eil}, J. Lie Theory \textbf{5} (1995), no.~2, 207--224.
  \MR{1389430 (97d:14074)}

\bibitem[Zha10]{MR2565534}
De-Qi Zhang, \emph{The {$g$}-periodic subvarieties for an automorphism {$g$} of
  positive entropy on a compact {K}{\"a}hler manifold}, Adv. Math. \textbf{223}
  (2010), no.~2, 405--415. \MR{2565534}

\end{thebibliography}
\end{document}